\documentclass{amsart}

\usepackage{tikz-cd}
\usepackage{amsmath}
\usepackage{amsfonts}
\usepackage{mathtools}
\usepackage{amsmath,amssymb,amsthm,mathrsfs}
\usepackage[utf8]{inputenc}
\usepackage{aeguill}
\usepackage{hyperref}
\usepackage{enumitem}

\theoremstyle{plain} 
\newtheorem{theorem}{Theorem} 
\numberwithin{theorem}{section}
\newtheorem*{theorem*}{Theorem}
\newtheorem{prop}[theorem]{Proposition}
\newtheorem{prop-def}[theorem]{Proposition-Definition}
\newtheorem{lemma}[theorem]{Lemma}
\newtheorem{coro}[theorem]{Corollary}
\newtheorem{conj}[theorem]{Conjecture}
\newtheorem{definition}[theorem]{Definition}
\theoremstyle{definition}

\theoremstyle{remark} 
\newtheorem{remark}[theorem]{Remark}

\theoremstyle{definition}

\newcommand{\wt}{\widetilde}

\newcommand{\be}{\begin{equation} }
\newcommand{\ee}{\end{equation} }

\usepackage{color}

\newcommand{\C}{\mathbb{C}} %is there a conflict somewhere?
\newcommand{\D}{\mathbb{D}}

\newcommand{\N}{\mathbb{N}}

\newcommand{\Q}{\mathbb{Q}}

\newcommand{\Z}{\mathbb{Z}}

\newcommand{\cF}{\mathcal{F}}

\newcommand{\cH}{\mathcal{H}}

\newcommand{\cK}{\mathcal{K}}
\newcommand{\cL}{\mathcal{L}}
\newcommand{\cM}{\mathcal{M}}

\newcommand{\cO}{\mathscr{O}}

\newcommand{\cT}{\mathcal{T}}

\newcommand{\cV}{\mathcal{V}}

\newcommand{\mfp}{\mathfrak{p}}

\newcommand{\bL}{\mathbf{L}}

\newcommand{\bs}{\mathbf{s}}

\newcommand{\gr}{\textup{gr}}

\newcommand{\sM}{\mathscr{M}}

\newcommand{\sE}{\mathscr{E}}

\newcommand{\sB}{\mathscr{B}}
\newcommand{\sO}{\mathscr{O}}
\newcommand{\sH}{\mathscr{H}}

\newcommand{\sA}{\mathscr{A}}

\newcommand{\sG}{\mathscr{G}}

\newcommand{\shD}{\mathscr{D}}
\newcommand{\sN}{\mathscr{N}}

\newcommand{\Spec}{\textup{Spec }}
\newcommand{\Specan}{\textup{Spec}^\an}

\newcommand{\Sym}{\textup{Sym}^\bullet}

\newcommand{\DR}{\textup{DR}}

\newcommand{\Mod}{\textup{Mod}}

\newcommand{\Coker}{\textup{Coker}}

\newcommand{\an}{\textup{an}}

\newcommand{\id}{\textup{id}}

\newcommand{\Kos}{\textup{Kos}}

\newcommand{\supp}{\textup{supp}}

\newcommand{\ba}{{\bf a}}
\def\bal{{\boldsymbol{\alpha}}}
\def\blamb{{\boldsymbol{\lambda}}}

\newcommand{\bk}{{\bf k}}

\newcommand{\bt}{{\bf t}}

\newcommand{\rel}{{\textup{rel}}}
\newcommand{\Ann}{{\textup{Ann}}}

\newcommand{\Ch}{\textup{Ch}}
\newcommand{\CC}{\textup{CC}}
\newcommand{\Ext}{\mathscr Ext}
\newcommand{\Rhom}{\mathscr Rhom}

\def\blam{{\boldsymbol{\lambda}}}

\title{Riemann-Hilbert correspondence for Alexander complexes}
\author{Lei Wu}
\address{Lei Wu, School of Mathematical Sciences, Zhejiang University, Hangzhou, 310058, P.R.China}
\email{leiwu23@zju.edu.cn}

\begin{document}


\setcounter{tocdepth}{1}

\numberwithin{equation}{section}

\begin{abstract}
We establish an explicit relative Riemann-Hilbert correspondence for Alexander complexes (also known as Sabbah specialization complexes) by using relative regular holonomic $\shD$-modules in an equivariant way, generalizing a classical result of Kashiwara and Malgrange for Deligne's nearby cycles. Using the correspondence and zero loci of Bernstein-Sato ideals%(under the general setting)
, we obtain a formula for the relative support of the Alexander complexes.
\end{abstract}
\maketitle

\tableofcontents

\section{Introduction}
Let $f$ be a holomorphic function on a complex manifold $X$ with $D$ the divisor of $f$ and let $U_x$ be a small open neighborhood of $x\in D$. Then we consider the fiber product diagram 
\[
\begin{tikzcd}
\widetilde U_x= (U_x\setminus D)\times_{\C^*}\widetilde \C^*\arrow[r]\arrow[d] & \widetilde \C^*\arrow[d,"exp"] \\
U_x\setminus D \arrow[r,"f"] & \C^* 
\end{tikzcd}
\]
where $exp\colon \widetilde \C^*\to \C^*$ is the universal cover of the punctured complex plane $\C^*$. The deck transformation induces a $\C[\pi_1(\C^*)]$-module structure on the compactly supported cohomology group $H^i_c(\widetilde U_x,\C)$, which is called the $i$-th (local) \emph{Alexander module} of $f$. Taking $U_x$ sufficiently small, the Alexander modules contain the information of the cohomology groups of the Milnor fibers around $x$ together with their monodromy action. 
As $x$ varies along $D$, all the local Alexander modules give a constructible complex of sheaves of $\C[\pi_1(\C^*)]$-modules, which recovers the Deligne nearby cycle of the constant sheaf  along $f$ (see \cite{Bry}).
Sabbah \cite{Sab90} made generalizations for a finite union of holomorphic functions and obtained what he called the \emph{Alexander complexes} (see \S \ref{sec:recalalexsab} for construction). 

Riemann-Hilbert correspondence between nearby cycles of regular holonomic $\shD$-modules and Deligne nearby cycles of complex perverse sheaves was constructed by Kashiwara \cite{KasV} and Malgrange \cite{MalV} by using $V$-filtrations along a single holomorphic function. A ``local" correspondence for complete specializations of Alexander complexes was constructed by Sabbah \cite[Theorem 5.1.2]{Sab90}; see also \cite{BG} for the algebraic ``local" case along a single regular function following the approach of Beilinson-Bernstein \cite{Bgluep}.

In this article, we establish a ``global" Riemann-Hilbert correspondence for the Alexander complexes in a functorial way  (Theorem \ref{thm:RHALX}). 
%The Alexander complexes are  the local Mellin transformations \cite{GL} of constructible complexes, which can also be treated as relative constructible complexes over affine tori (see \S \ref{subsec:AcomSab}). 
%However, there is no well-defined notion of ``normalized'' multi indexed $V$-filtrations along a finite union of holomorphic functions in general (based on the work of Sabbah \cite{Sab,Sab2}; see also \cite[Remark 4.22]{Wuch}). Thus, one would not be able to generalize the method of Kashiwara and Malgrange to construct the correspondence for  
%Alexander complexes in general.
Our approach to the ``global" correspondence can be seen as a hybrid of the method of Sabbah \cite{Sab,Sab2} and that of Beilinson-Bernstein \cite{Bgluep}. 
%The new inputs include: 
%\begin{enumerate}
%    \item the analytic sheafification of relative $\shD$-modules over commutative rings $R$,
%    \item the sheafification of constructible complexes of sheaves of $R$-modules,
%    \item equivariant sheaves as in \cite{BL}.
%\end{enumerate}
%the analytic sheafification of both algebraic relative $\shD$-modules and constructible sheaves of $R$-modules and equivariant sheaves. 
Our construction also relies on the work of Maisonobe \cite{Mai} in the study of Bernstein-Sato ideals by using relative holonomic $\shD$-modules over algebraic affine spaces and its development in \cite{WZ,BVWZ,BVWZ2}, and on the theory of analytic relative holonomic $\shD$-modules and relative constructible complexes developed in a series of fundamental papers \cite{FSO, FS,FF18, FFS19}. Very recently, the relative regular Riemann-Hilbert correspondence has been fully established in \cite{FMSRRR}. Our main results (Theorem \ref{thm:RHALX} and \ref{thm:RHALXhigh}) can thus be seen as explicit examples of the relative Riemann-Hilbert correspondence in the scope of their general theory.
Compared to \cite{FFS19, FMSRRR}, the new inputs we need to establish the explicit correspondence include a GAGA-type principle between analytic and algebraic relative $\shD$-modules and between algebraic and analytic relative constructible sheaves (see Appendix \ref{sec:shffm}), and $G$-equivariant relative $\shD$-modules (see \S \ref{subsec:gluemmex}).

% the analytic sheafification for $\shD_R$-modules (see \S \ref{subsec:ashfdmo}) is based on the study of algebraic relative (logarithmic) holonomic $\shD$-modules in \cite{WZ,BVWZ,BVWZ2,Wuch}. The analytic relative holonomic $\shD$-modules and relative constructible sheaves have been studied in the fundamental papers \cite{FSO, FS, FFS19}.  The sheafification for $\shD_R$-modules can be treated as the GAGA principle between analytic and algebraic relative $\shD$-modules, whilethe sheafification of constructible sheaves of $R$-modules (see \S \ref{subsec:ansh} and \S \ref{subsec:anshcon}) can be treated as an algebraic simplification of the relative constructible sheaves defined in \cite{FSO}. Let us emphasis that the construction of the correspondence in Theorem \ref{thm:RHALX} relies on the sheafification procedure. Our proof of the correspondence also relies on the result of relative constructibility of the de Rham complexes of relative holonomic $\shD$-modules in \cite{FSO}. 

%Let us briefly explain the terminology of sheafification. When the base space $X$ is a point, the sheafification of $R$-modules is just the functor transforming $R$-modules into quasi-coherent $\sO$-modules on $\Spec R$. For general base spaces, it is the relative analogue for complexes of sheaves of $R$-modules. One then make the base change onto $\Specan R$ to obtain the analytic one for $R$ commutative finite generated $\C$-algebras. See Remark \ref{rmk:necansh} for the reason why the analytic sheafification is necessary in proving Theorem \ref{thm:RHALX} and Theorem \ref{thm:RHALXhigh}.

\subsection{Riemann-Hilbert correspondence for Alexander complex}
Let $f_i$ be holomorphic functions on a complex manifold $X$ for $i=1,\dots,r$ with $D_i$ the divisor of $f_i$. We write $F=(f_
1,\dots,f_r)$, and $D=\sum_i D_i$. Suppose that $\cM$ is a (left) holonomic $\shD_X$-module. We write by
\[\cM(*D)=\cM\otimes_{\sO}\sO_X(*D)\]
the algebraic localization of $\cM$ along $D$, where
$\sO_X(*D)=\bigcup_{k\in \Z}\sO_X(kD)$
is the sheaf of meromorphic functions with poles along $D$.

To establish the Riemann-Hilbert correspondence for Alexander complexes of Sabbah, we first construct the relative maximal (resp. minimal) extensions of $\cM$ along $F$, denoted by $\cM(*D^{(r)}_F)$ (resp. $\cM(!D^{(r)}_F)$), which are both relative holonomic $\shD_{X\times\C^r/\C^r}$-modules where $\shD_{X\times\C^r/\C^r}$ is the sheaf of relative holomorphic differential operators with respect to the projection $pr\colon X\times\C^r\to \C^r$ (see Definition \ref{def:relcohhol} for relative holonomicity). More precisely, $\cM(*D^{(r)}_F)$ is the analytic sheafification of the $\shD_X[\bs]$-module 
\[\cM(*D)\otimes_\C \C[\bs]\cdot F^\bs,\quad \bs=(s_1,\dots s_r) \textup{ the algebraic coordinates of }\C^r\]
and $\cM(!D^{(r)}_F)$ is the $\shD_{X\times\C^r/\C^r}$-dual of $(\D\cM)(*D^{(r)}_F)$, where
$$F^\bs=\prod_i f_i^{s_i}$$
is a formal symbol and $\D(\cM)$ is the $\shD_X$-dual of $\cM$. See \S \ref{subsec:!*relex} for details. It is worth mentioning that when $\cM=\sO_X$, $\sO_X(*D_F^{(r)})$ gives a non-trivial example of relative Deligne meromorphic extensions of relative local systems (see \S \ref{subsec:univrellocs} and also \cite[\S2.2]{FS17} and \cite[\S8]{NSI}).

The key property for the maximal and minimal extensions is that there exists a natural inclusion (Lemma \ref{lm:!*incl})
\be\cM(!D^{(r)}_F)\hookrightarrow \cM(*D^{(r)}_F),\ee
which is a ``global'', sheafified and higher dimensional generalization of a classical result of Beilinson and Bernstein (see \cite[Proposition 3.8.3]{Gil} and also \cite{BG}). 
%The ``local" maximal and minimal extension of Beilinson and Bernstein in \cite{BG} are very useful in representation theory, for instance they can be used to construct tilting sheaves \cite{Cambp}. It would be interesting to know if $\cM(!D^{(r)}_F)$ and $\cM(*D^{(r)}_F)$ have related applications in representation theory.
We then define a relative holonomic $\shD$-module on $X\times \C^r$,
\[\Psi_{F}(\cM)=\frac{\cM(*D^{(r)}_F)}{\cM(!D^{(r)}_F)}.\]

We denote by
\[Exp\colon \C^r\rightarrow (\C^*)^r, (\alpha_1,\dots,\alpha_r)\mapsto(exp(-2\pi \sqrt{-1}\alpha_1),\dots,exp(-2\pi \sqrt{-1}\alpha_r))\] 
the universal covering of $(\C^*)^r$ and by 
\[\pi=(\textup{id},Exp)\colon X\times\C^r\rightarrow X\times (\C^*)^r\]
the induced map. We write $G=\pi_1((\C^*)^r)$, the fundamental group of $(\C^*)^r$. The universal covering makes $X\times \C^r$ a $G$-space with the quotient $G\backslash X\times \C^r=X\times (\C^*)^r$.
The operation 
\[t_i(F^\bs)=f_iF^\bs\]
with $t_i$ representing the (counterclockwise) loops around the puncture of each factor $\C^*$ of $(\C^*)^r$,
makes $\cM(*D^{(r)}_F)$, $\cM(!D^{(r)}_F)$ and $\Psi_{F}(\cM)$ all $G$-equivariant sheaves in the sense of \cite[Part I.0.2.]{BL} (see Theorem \ref{thm:!*Gequiv}, Theorem \ref{thm:gnbGeq} and Remark \ref{rmk:reldsteqdb}(1)). 
\begin{theorem}\label{thm:RHALX}
With the notation as above, if $\cM$ is a regular holonomic $\shD_X$-module, then
we have a natural quasi-isomorphism
\[\pi_*^G\big(\DR_{X\times\C^r/\C^r}(\Psi_{F}(\cM))\big)\simeq \widetilde\psi_{F}(\DR_X(\cM)),\]
where $\pi_*^G$ is the equivariant direct image functor (see \cite[Part I.0.3.]{BL}), $\DR_{X\times\C^r/\C^r}$ is the relative de Rham functor on $X\times\C^r$ over $\C^r$ and  $\widetilde\psi_{F}(\DR_X(\cM))$ is the analytic sheafification of the Alexander complex $\psi_{F}(\DR_X(\cM))$.
\end{theorem}
%The $\C[G]$-module structure of $\pi_*^G\big(\DR_{X\times\C^r/\C^r}(\Psi_{F}(\cM))\big)$ is induced by\[Exp_*^G(\sO_{\C^r})\simeq \sO_{(\C^*)^r}\]and the natural inclusion\[\C[G]\simeq \C[t_1^{\pm1},\dots,t_r^{\pm1}]\hookrightarrow \sO_{(\C^*)^r},\]where in the inclusion $\C[G]$ is treated as a constant sheaf on $(\C^*)^r$. 
Theorem \ref{thm:RHALX} is related to the local comparision for Alexander complexes of Sabbah \cite[Theorem 5.1.2]{Sab90}. To be more precise, one can take the (complete) localization of the quasi-isomorphism in Theorem \ref{thm:RHALX} at a point $\blam\in(\C^*)^r$ to obtain the local comparison for $\psi_{F}(\DR_X(\cM))$. 
%(without the quasi-unipotent assumption in \cite[Theorem 5.1.2]{Sab90}).
Notice that localizing $\pi_*^G\big(\DR_{X\times\C^r/\C^r}(\Psi_{F}(\cM))\big)$ at $\blam$ is equivalent to localizing $\DR_{X\times\C^r/\C^r}(\Psi_{F}(\cM))$ at $\bal$ for every $\bal\in Exp^{-1}(\blam)$, since $G$ acts freely on $\C^r$ with $G\backslash\C^r=(\C^*)^r$. In particular, when $r=1$, localizing the quasi-isomorphism in Theorem \ref{thm:RHALX} at $\lambda\in \C^*$ gives the comparison between the $\alpha$-nearby cycle of $\cM$ and the $\lambda$-nearby cycle of $\DR(\cM)$ along $f$ for every $\alpha\in Exp^{-1}(\lambda)$ (see \cite{KasV, MalV, BG} and also \cite{Wubf}). Notice that the $G$-action in Theorem \ref{thm:RHALX} explains why the correspondence between $\alpha$-nearby cycles of regular holonomic $\shD$-modules and $\lambda$-nearby cycles of perverse sheaves along a single holomorphic function is $\Z$-to-1.

%The restriction of Theorem \ref{thm:RHALX} on $X\times V$ becomes a version of relative regular Riemann-Hilbert correspondence over $V$,  with $V$ a small open subset of $\C^r$ such that $\pi|_{X\times V}$ is an isomorphism. 
Since the relative de Rham functor is $G$-equivariant (see \S \ref{subsec:deRhamequiv}), we have a natural isomorphism
\[\DR_{X\times(\C^*)^r/(\C^*)^r}\big(\pi_*^G(\Psi_F(\cM))\big)\simeq \pi_*^G\big(\DR_{X\times\C^r/\C^r}(\Psi_{F}(\cM))\big).\]
Then for regular holonomic $\shD_X$-modules $\cM$, $\DR_{X\times(\C^*)^r/(\C^*)^r}(\bullet)$ gives a relative Riemann-Hilbert correspondence on $X\times(\C^*)^r$ over $(\C^*)^r$
\[\begin{tikzcd}
\pi_*^G(\Psi_F(\cM)) \arrow[r] & \wt\psi_F(\DR_X(\cM))\arrow[l]
\end{tikzcd}
\]
with ``$\longleftarrow$" induced by the Riemann-Hilbert correspondence for regular holonomic $\shD$-modules on $X$.
%See also \cite{FFS19} for a relative Riemann-Hilbert correspondence over curves and \S \ref{subsec:questionalgrrh} for a related open question.
%It seems that Theorem \ref{thm:RHALX} is not in its final form. It would be interesting to ask whether there exists a similar correspondence for relative $\shD$-modules over for example principle $G$-bundles for $G$ a Lie group in general. 
%One of the major applications of the Bernstein-Beilinson construction of nearby cycles is the proof of Jantzen’s conjectures \cite{BBJz}. Other related applications include \cite{Gil1,Gil3}. 

Theorem \ref{thm:RHALX} can be further refined to a comparision for higher-codimensional Sabbah specialization complexes (see \S \ref{subsec:AcomSab}) as follows by using sheaves of (algebraic) local cohomology.
\begin{theorem}\label{thm:RHALXhigh}
In the situation of Theorem \ref{thm:RHALX}, for every subset $I\subseteq \{1,2,\dots,r\}$ we have quasi-isomorphisms
\[\begin{aligned}
&\pi_*^G\big(\DR_{X\times\C^r/\C^r}\D R\Gamma_{[D_I\times\C^r]}\big((\D\cM)(!D^{(r)}_F)\big)\big)\\
%&\simeq \DR_{X\times(\C^*)^r/(\C^*)^r}\D \big(\pi_*^G(R\Gamma_{[D_I\times\C^r]}((\D\cM)(!D^{(r)}_F)))\big)\\ 
&\simeq \DR_{X\times(\C^*)^r/(\C^*)^r}\D R\Gamma_{[D_I\times(\C^*)^r]}\big(\pi_*^G\big((\D\cM)(!D^{(r)}_F)\big)\big)\\ 
&\simeq \widetilde\psi_{D_I}(\DR_X(\cM)),
\end{aligned}
\]
where $\psi_{D_I}(\DR_X(\cM))$ is the Sabbah specialization complex along $D_I=\bigcap_{i\in I} D_i$.
\end{theorem}
%The above theorem has a direct consequence that when $X$ is a smooth algebraic variety over $\C$ and $\cM$ is an algebraic regular holonomic $\shD_X$-module (for intance $\cM=\sO_X$), $\pi_*^G\big(R\Gamma_{[D_I\times\C^r]}(\cM(!D^{(r)}_F))\big)$ is algebraic for every $I$ (see Remark \ref{rmk:algshd}), since $\widetilde\psi_{D_I}(\DR_X(\cM))$ is so. But $R\Gamma_{[D_I\times\C^r]}(\cM(!D^{(r)}_F))$ is only analytic.

\subsection{Relative Support}
We now discuss the relative supports of $\Psi_F(\cM)$ and the Alexander complex of $\DR(\cM)$ (see Definition \ref{def:relsupp} and Proposition-Definition \ref{prop-def:relsupp}). Applying Theorem \ref{thm:RHALX}, 
%by Lemma \ref{lm:relconstrnakayam} 
we obtain:
\begin{coro}
\label{thm:suppRH}
Let $\cM$ be a regular holonomic $\shD_X$-module on a complex manifold $X$ and let $F=(f_1,\dots,f_r)$ be an $r$-tuple of holomorphic functions on $X$. Then 
\[\supp_{(\C^*)^r}\psi_F(\DR(\cM))=Exp\big(\supp_{\C^r}\Psi_F(\cM)\big)\]
and
\[\supp_{\C^r}\Psi_F(\cM)=Exp^{-1}\big(\supp_{(\C^*)^r}\psi_F(\DR(\cM))\big).\]
\end{coro}
The action of $G$ on $\C^r$ induces an action of $G$ on its algebraic coordinate ring $\C[\bs]$, i.e.
\[t_i\cdot s_j=
\begin{cases}
s_j+1, & \textup{ if } i=j \\
s_j, &\textup{ if } i\not=j.
\end{cases}
\]
Using Bernstein-Sato ideals (see \S \ref{def:bsideal} and \S \ref{subsec:BSideal}), we obtain a geometric description of the relative support of $\Psi_F(\cM)$:
\begin{theorem}\label{thm:relsd}
For each pair $(\cM, F)$ as in Corollary \ref{thm:suppRH}, locally on a relatively compact open subset \footnote{The condition  ``locally on a relatively compact open subset" (here and elsewhere it appears in this article) is to ensure finiteness. Thus, if we are in the algebraic category, then this condition can usually be eliminated.} of $X$ there exist finite sets
\[S(F,\cM)\subseteq \N_{\ge0}^r \textup{ and } \kappa(L)\subseteq \C \textup{ for each } L\in S(F,\cM)\]
such that 
\[\supp_{\C^r}\Psi_F(\cM)=\bigcup_{g\in G}\bigcup_{L\in S(F,\cM)}\bigcup_{\alpha\in \kappa(L)}(g\cdot(L\cdot\bs+\alpha)=0)\subseteq \C^n,\]
where $(g\cdot(L\cdot\bs+\alpha)=0)$ denotes the divisor inside $\C^r$ defined by $g\cdot(L\cdot\bs+\alpha)=0$.
\end{theorem}
In the theorem above, $S(F,\cM)$ is the union of all (primitive) slopes of the codimension-one components of the zero locus of a Bernstein-Sato ideal of $\cM$ along $F$, and for every $L\in S(F,\cM)$ the finite set $\kappa(L)$ is the set of all $\alpha\in\C$ such that $L\cdot \bs+\alpha=0$ defines an irreducible codimension-one component of the zero locus of the Bernstein-Sato ideal. See \S \ref{subsec:BSideal} for details. 

Corollary \ref{thm:suppRH} and Theorem \ref{thm:relsd} together give:
\begin{coro}\label{cor:suppam}
Locally on a relatively compact open subset of $X$, 
\[\supp_{(\C^*)^r}\psi_F(\DR(\cM))=\bigcup_{L\in S(F,\cM)}\bigcup_{\alpha\in \kappa(L)} (\bt^L=exp(2\pi\sqrt{-1}\alpha))\subseteq (\C^*)^r,\]
where $\bt^L=\prod_i t_i^{l_i}$ with $L=(l_1,\dots,l_r)$ and $(\bt^L=exp(2\pi\sqrt{-1}\alpha))$ is the divisor defined by $\bt^L=exp(2\pi\sqrt{-1}\alpha)$ in $(\C^*)^r$.
In particular, $\supp_{(\C^*)^r}\psi_F(\DR(\cM))$is a finite union of translated codimensional-one subtori inside $(\C^*)^r$.
\end{coro}
The study of $\supp_{(\C^*)^r}\psi_F(\DR(\cM))$ was initiated by Sabbah \cite{Sab90}, where the support is shown to be included in a union of translated subtori. 
%See also \cite[Theorem 1.3]{BLSW} for the case $\cM=\sO_X$ from a purely topological approach. 
In the case $\cM=\sO_X$, it is proved in \cite[Theorem 1.3]{BLSW} that $\supp_{(\C^*)^r}\psi_F(\C_X)$ is a finite union of torsion translated codimensional-one subtori from a purely topological approach. Corollary \ref{cor:suppam} gives a precise description of $\supp_{(\C^*)^r}\psi_F(\DR(\cM))$ in general.  
If additionally $\cM$ is quasi-unipotent along $F$ (for instance $\cM=\sO_X$) , then $\kappa(L)\subseteq \Q$ for every $L$ and hence $\supp_{(\C^*)^r}\psi_F(\DR(\cM))$ is a finite union of torsion translated subtori inside $(\C^*)^r$. In consequence, Theorem \ref{thm:RHALX}, Corollary \ref{thm:suppRH}, Theorem 
\ref{thm:relsd} and Corollary \ref{cor:suppam} together give a generalization of \cite[Theorem 1.5.1]{BVWZ} as well as \cite[Theorem 2 and Theorem 3]{BG}.

As mentioned above, $\supp_{(\C^*)^r}\psi_F(\C_X)$ has a purely topological interpretation by using (local) cohomology jumping loci \cite{BLSW}. But the cohomology jumping loci (of quasi-projective varieties or analytic germ complements) might contain torsion-translated subtori of higher codimension \cite{BWcjl, BW}. 
%Motivated by Corollary \ref{thm:suppRH}, o
Then one can naturally ask how to give a $\shD$-module interpretation of such lower dimensional loci. Along this line, we propose a linearity conjecture (Conjecture \ref{conj:linarityextsupp}) by using the codimension filtration of Gabber-Kashiwara in the relative setting (see Appendix \ref{subsec:relcodimfil}). The linearity conjecture is interesting because it would imply a conjecture of Budur \cite{Budur} that the zero locus of the Bernstein-Sato ideal of $F$ is a finite union of translated linear subspaces of $\C^r$ defined over $\Q$ (see Proposition \ref{prop:myconjimplybudur}).

\subsection{Relative characteristic cycle and monodromy zeta function}\label{subsection:rcczetaindex}
Our next results are about understanding the irreducible divisor $$(g\cdot(L\cdot \bs+\alpha)=0)\subseteq \C^r$$ inside $\supp_{\C^r}\Psi_F(\cM)$ and their relations with relative characteristic cycles and the monodromy zeta functions. 

We now assume $\cM$ a regular holonomic $\shD_X$-module and $F=(f_1,\dots,f_r)$ an $r$-tuple of holomorphic functions on a complex manifold $X$. 
\begin{theorem}\label{thm:relccgny}
With assumptions above , we have the relative characteristic cycles
\[\CC^\rel(\cM(*D_F^{(r)}))=\CC^\rel(\cM(!D_F^{(r)}))=\CC(\cM(*D))\times \C^r\subseteq T^*X\times\C^r,\]
where $\CC^\rel$ and $\CC$ denote the relative characteristic cycle and the absolute one respectively and similarly for the characteristic variety $\Ch$.
%$\CC(\cM(*D))$ is the characteristic cycle of $\cM(*D)$.
Moreover, locally on a relatively compact open subset $W$ of $X$, the relative characteristic cycle of $\Psi_F(\cM)$ is an infinite sum
\[\CC^\rel(\Psi_F(\cM))=\sum_{g\in G}\sum_{L\in S(F,\cM)}\sum_{\alpha\in \tilde\kappa(L)} \Lambda_{L,\alpha}\times (g\cdot(L\cdot \bs+\alpha)=0)\]
where each $\Lambda_{L,\alpha}$ is a Lagrangian cycle supported on $\Ch(\cM(*D))$ $($over $D\cap W)$, and $\tilde\kappa(L)$ is $\kappa(L)$ modulo $L$-equivalence $($Definition \ref{def:slopkappa} $)$. 
\end{theorem}
The relative characteristic cycle of $\Psi_F(\cM)$ is a symmetric infinite sum because of the G-action on $\Psi_F(\cM)$ (if a $G$-equivariant sheaf is supported on a point, then it is supported on the $G$-orbit of the point). If $\wt W$ is a relatively compact open subset of $X\times\C^r$, then $$\big(\sum_{g\in G}\sum_{L\in S(F,\cM)}\sum_{\alpha\in \tilde\kappa(L)} \Lambda_{L,\alpha}\times (g\cdot(L\cdot s-\alpha)=0)\big)\big|_{\wt W}$$ becomes a finite sum since $X\times (g\cdot(L\cdot s+\alpha)=0)$ intersect $\wt W$ for only finite many $g\in G$. See Proposition \ref{prop:lambdaLalf1dim} for formulas of $\Lambda_{L,\alpha}$ in terms of the characteristic cycles of very general choices of one-dimensional nearby cycles associated to $\cM$. %By Theorem \ref{thm:RHALX}, in the language of Kashiwara and Schapira \cite[Theorem 8.5.5]{KSbook}, the support of $\pi_*(\CC^\rel(\Psi_F(\cM)))$ gives the singular support of $\psi_F(\DR(\cM))$. 

The relative characteristic cycle formula of $\cM(*D^{(r)}_F)$ in Theorem \ref{thm:relccgny}, as well as the irregular case in Theorem \ref{thm:relholmaxexcc}, has an application to the study of constructibility of the logarithmic de Rham complexes of lattices of holonomic modules \cite[Theorem 1.1]{Wuch}.

We now fix an arbitrary point $x\in X$. By Theorem \ref{thm:relsd} we can focus on $S(F,\cM)$ and $\tilde\kappa(L)$ (or $\kappa(L)$) for $L\in S(F,\cM)$ locally around $x$. By Theorem \ref{thm:relccgny}, both $S(F,\cM)$ and $\tilde\kappa(L)$ (or $\kappa(L)$) are constructible as $x$ varies on $X$ and hence so are $\supp_{(\C^*)^r}\psi_F(\DR(\cM))$ and $\supp_{\C^r}\Psi_F(\cM)$. More precisely, we fix a finite Whitney stratification
\[X=\bigsqcup_\beta X_\beta \textup{ such that } \Ch(\cM(*D))\subseteq \bigsqcup_\beta T_{X_\beta}^*X,\]
where $T_{X_\beta}^*X$ is the conormal bundle of the smooth strata $X_\beta$.
%see \cite[\S 3]{Gil} for the formula of $\CC(\cM(*D))$. 
Then $S(F,\cM)$ and $\tilde\kappa(L)$ for $L\in S(F,\cM)$ are locally constant with respect to the above stratification. Moreover, for $L\in S(F,\cM)$ and $\alpha\in \tilde\kappa(L)$ locally around $x$, we have 
\[\Lambda_{L,\alpha}=\sum_\beta m_\beta(L,\alpha) \overline{T^*_{X_\beta}X}, \quad m_\beta(L,\alpha)\in \Z_{\ge 0}.\]

Next, we give a formula for the multiplicity $m_\beta(L,\alpha)$ by using the monodromy zeta function. By construction, $\psi_{F}(\DR(\cM))$ is a $\C$-constructible complex of sheaves of $\C[G]$-modules (see \cite[\S 8.5]{KSbook} for the definition of $\C$-constructibility). Hence, for each generic point $q$ of Supp$_{(\C^*)^r}(\psi_F(\DR(\cM)))$
\[\psi_{F,q}(\DR(\cM))=\psi_{F}(\DR(\cM))\otimes_{\C[G]}\C[G]_{q}\]
is a $\C$-constructible complex of sheaves of $\C[G]_{q}$-modules, where $\C[G]_{q}$ is the localization of $\C[G]$ at the prime ideal $q$. We define 
\[\chi_x(\psi_{F}(\DR(\cM)),q)=\sum_i (-1)^i \textup{lg}(\cH^i\psi_{F,q}(\DR(\cM))|_{x})\]
where lg denotes the length function, 
and 
\[\zeta_x(\psi_F(\DR(\cM)))=\sum_q \chi_x(\psi_{F}(\DR(\cM)),q)\bar q\]
where $q$ goes over the set of all the generic points of  Supp$_{(\C^*)^r}(\psi_F(\DR(\cM)))$ (since Supp$_{(\C^*)^r}(\psi_F(\DR(\cM)))$ is an algebraic closed subset of $(\C^*)^r$ by Corollary \ref{cor:supprelccsh}).
By $\C$-constructibility, $\chi_x(\psi_{F}(\DR(\cM),q)$ is a $\Z$-valued constructible function on $X$ and $\zeta_x(\psi_F(\DR(\cM)))$ is a constructible function on $X$ valued in the abelian group of algebraic subvarieties in $(\C^*)^r$. The constructible function $\zeta_x(\psi_F(\DR(\cM)))$ is the so-called \emph{monodromy zeta function} of $\psi_F(\DR(\cM))$ (cf. \cite[\S 2.5]{Sab90}).

We denote by $q(L,\alpha)\in \Spec \C[G]$ the generic point of $(\bt^L-exp(2\pi\sqrt{-1}\alpha)=0)$ for $L$ and $\alpha$. With the help of Corollary \ref{cor:suppam}, we obtain a precise formula:
\be\zeta_x(\psi_F(\DR(\cM)))=\prod_{L\in S(F,\cM)}\prod_{\alpha\in\tilde \kappa(L)} (\bt^L-exp(2\pi\sqrt{-1}\alpha))^{\chi_x(\psi_{F}(\DR(\cM)),q(L,\alpha))}, \ee
where we make the constructible function valued in rational functions. 

Using a local index formula of Kashiwara-Dubson-Ginsburg \cite[Theorem 8.2]{Gil}, we get a formula for $\chi_x(\psi_{F,q}(\DR(\cM)))$ and $m_{\beta_1}(L,\alpha)$:
\be \label{eq:genericindzetaf}
m_{\beta_1}(L,\alpha)=\sum_{ X_{\beta_1}\subseteq \bar X_\beta }(-1)^{\dim X_\beta}c(X_{\beta_1},X_\beta)\chi_{x_\beta}(\psi_{F}(\DR(\cM)),q(L,\alpha))),\ee
where
\[
c(X_{\beta_1},X_\beta)=
\begin{cases}
\chi^{top}(B_{x_{\beta_1}}\cap X_\beta\cap H), & \textup{ if } X_{\beta_1}\not=X_{\beta} \\
1, &\textup{ if } X_{\beta_1}=X_{\beta}.
\end{cases}
\]
with $B_{x_{\beta_1}}$ a small polydisc open neighborhood inside $X$ of some $x_{\beta_1}\in X_{\beta_1}$, $H$ a linear subspace of $X$ sufficiently close to $x_{\beta_1}$ of codimension $\dim X_{\beta_1}+1$ and $\chi^{top}$ denoting the topological Euler characteristic. 

\subsection{Outline of the paper} Section \ref{sec:*!ext} is about the construction of maximal and minimal extensions under the relative setting and their $G$-equivariance. In Section \ref{sec:recalalexsab}, we recall the construction of the Alexander complexes of Sabbah. In Section \ref{sec:comparison}, we discuss comparisons in the sense of relative Riemann-Hilbert correspondence and prove Theorem \ref{thm:RHALX} and Theorem \ref{thm:RHALXhigh}. In the appendix, we discuss the relative sheafification functor over commutative rings and discuss the properties of algebraic and analytic relative supports. Results in the appendix is standard and more or less well-known; we include it for completeness.  
%We then write (locally on a relatively compact open subset of X)
%\[\Lambda_{L,\alpha}=\sum_{j\in J} m_j\Lambda_{L,\alpha}^j\subseteq T^*X\]
%with each $\Lambda_{L,\alpha}^j$ an irreducible Lagrangian subvariety of $T^*X$ for a finite index set $J$. 

%We recall that the zeta-function of the Alexander module $\phi_F(\DR(\cM))$ (in the sense of Sabbah \cite[\S 2.5]{Sab90}):
%\[\zeta_x(\psi_F(\DR(\cM)))=\prod_q  L_q^{\chi_x(\psi_F(\DR(\cM))_q)}\]
%where $q$ goes over all the generic point of the irreducible component of $\supp_{(\C^*)^r}\psi_F(\DR(\cM))$ around $x$, and $L_q$ is the generator of $q$ (notice that the (algebraic) ideal of each component is principal).

%\subsection{Notations} For $\shD$-modules and relative $\shD$-modules, we follow the notations and the terminology in \cite[\S 2]{Wuch}; for constructible sheaves on complex manifolds, we follow \cite[\S 8.5]{KSbook}; for equivariant sheaves we follow \cite{BL}.

\subsection*{Acknowledgement}
The author would like to thank Nero Budur, Linquan Ma, Mircea Musta\c{t}\v{a}, Ruijie Yang and Peng Zhou for useful discussions and comments and for answering questions. We are grateful to an anonymous referee for enormous very useful suggestions and corrections and for the change of the structure of the paper.   

%The author was supported by the grant Methusalem METH/15/026 from KU Leuven.

\section{Relative maximal and minimal extensions along hypersurfaces}\label{sec:*!ext}
In this section, we construct the maximal and minimal extensions for $\shD$-modules under the relative setting analogous to the $^pj_!$ and $^pj_*$ extensions of affine (or Stein) open inclusions $j$ for perverse sheaves \cite{BBDG}. 

\subsection{Notations}\label{subsec:3notations}
We introduce notation for this whole section and we refer to the appendix for the notation not defined here. Let $X$ be a complex manifold with $\dim_\C X=n$. For some fixed $r\in \Z_{>0}$, we consider the morphism 
\[\pi=(\textup{id},Exp)\colon X\times\C^r\rightarrow X\times (\C^*)^r\]
where 
\[Exp\colon \C^r\rightarrow (\C^*)^r, (\alpha_1,\dots,\alpha_r)\mapsto(exp(-2\pi \sqrt{-1}\alpha_1),\dots,exp(-2\pi \sqrt{-1}\alpha_r))\] 
is the universal cover of $(\C^*)^r$. We write $G=\pi_1((\C^*)^r)$, the fundamental group of $(\C^*)^r$. The deck transformation of $Exp$ makes $X\times \C^r$ a free $G$-space (we always use the discrete topology of $G$). We also write \[R=\Sym (\C^r)^\vee\simeq \C[\bs]\]
the algebraic coordinate ring of $\C^r$, where 
$\Sym (\C^r)^\vee$ is the symmetric algebra of the dual vector space $(\C^r)^\vee$ and the isomorphism is induced from using the standard basis of $\C^r$. Then $G$ acts on $R$ induced by the $G$-action on $(\C^r)^\vee$. In algebraic coordinates, the $G$-action on $R$ is given by
\be\label{eq:gactiononR}
t_i\cdot s_j=
\begin{cases}
s_j+1, & \textup{ if } i=j \\
s_j, &\textup{ if } i\not=j.
\end{cases}
\ee
where $t_i$ represents the (counterclockwise) loops around the puncture of each factor $\C^*$ of $(\C^*)^r$ and such $t_i$ induce an isomorphism $G\simeq \Z^r$.

Let $F=(f_1,\dots,f_r)$ be an $r$-tuple of holomorphic functions on $X$ with $D_i$ the divisor of $f_i$ and let $\cM$ be a holonomic $\shD_X$-module. We write $D=\sum_i D_i$ and by
\[\cM(*D)=\cM\otimes_{\sO}\sO_X(*D)\]
the algebraic localization of $\cM$ along $D$, where
$\sO_X(*D)=\bigcup_{k\in \Z}\sO_X(kD)$
is the sheaf of meromorphic functions with poles along $D$.  We write 
\[\sM=\cM(*D)\otimes_\C\C[\bs]\cdot F^\bs.\]
We have natural actions
\be\label{eq:tildMalt}
\theta\cdot (m\cdot F^\bs)=(\theta\cdot m)\cdot F^\bs+m\cdot\sum_{i=1}^r \frac{s_i\theta(f_i)}{f_i}\cdot F^\bs\ee
for every section $m$ of $\cM(*D)$ and for holomorphic vector fields $\theta$ on $X$, which makes $\sM$ a $\shD_{X,R}$-module. 
%Or equivalently, one can define $\wt\sM$ to be $p^*(\cM(*D))$ (as $\sO$-modules) with the twisted $\wt\shD_{X,\C[\bs]}$-module structure given by \eqref{eq:tildMalt}.
Since $\cM(*D)$ is holonomic, by \cite[Theorem 3.1]{MalIr} one can assume 
\[\cM(*D)=\shD_X\cdot\cM_0\]
for $\cM_0$ a coherent $\sO_X$-submodule of $\cM(*D)$. We then fix such an $\cM_0$ throughout the remaining of this paper 
and define the $\shD_{X,R}$-submodules
\[\sN_k=\shD_{X,R}(\cM_0\cdot F^{\bs-\bk})\subseteq \sM\]
generated by $\cM_0\cdot F^{\bs-\bk}$ with $\bk=(k,k,\dots,k)$ for $k\in\Z$. By construction, $\shD_{X,R}(\cM_0\cdot F^\bs)$ is coherent over $\shD_{X,R}$ but $\sM$ might not be so, and 
\[\sN_{k_1}\subseteq \sN_{k_2} \textup{ for $k_1\le k_2$ and } \lim_{k\to\infty} \sN_k=\sM.\]

\subsection{$G$-action and translation on $\sM$}
The $G$-action on $R$ induces a $G$-action on $\sM$.
In algebraic coordinates, the $G$-action on $\sM$ is given by
\[t_i\cdot P(\bs)(m\cdot F^\bs)=P(\bs+e_i)(m\cdot F^{\bs+e_i})\]
for $P(\bs)\in \shD_{X,R}$, where $e_i\in\Z^r$ is the $i$-th unit vector. The $G$-action on $\sM$ induces an operation $g\cdot\sN_k\subset \sM$, i.e.
\[t_i\cdot\sN_k=\shD_X[\bs]\cdot F^{\bs-\bk+e_i}.\]

%\begin{lemma}
%The above $G$-action on $\sM$ makes $\wt\sM$ a $G$-equivariant sheaf on  $X\times \C^r$. Moreover, for $P\in \shD_X$, the $P$-action on $\wt\sM$ is $G$-equivariant.
%\end{lemma}
%\begin{proof}
%We use the definition of $G$-equivariant sheaves in \cite[Part I, 0.2]{BL}. The identity $g\cdot\sM=\sM$ for $g\in G$ gives the isomorphism $\theta$ in \emph{loc. cite.} with the cocycle condition tautologically satisfied. 
%Since $P$-action and $t_i$-action commute, the second statement follows.  
%\end{proof}

Since $\C^r$ is homogeneous, for every $\bal\in \C^r$ the translation by $\bal$ induces an isomorphism 
\[\tau_\bal\colon X\times\C^r\to X\times\C^r\quad (x,\bs)\mapsto (x,\bs+\bal).\]
%operation 
%\[\tau_{\bal}\sN_k=\shD_{X,R}(\cM_0\cdot F^{\bs-\bal+\bk})\subseteq \tau_{\bal}\sM=\cM(*D)\otimes_\C\C[\bs]\cdot F^{\bs-\bal},\]
Substituting $\bs$ by $\bs+\bal$ gives an $\shD_X$-linear isomorphism (but not $\C[\bs]$-linear)
\be\label{eq:subtrop}
\sN_k\simeq \shD_{X}[\bs+\bal](\cM_0\cdot F^{\bs+\bal-\bk})
\ee
which further induces an $pr_X^{-1}\shD_X$-linear isomorphism
\be\label{eq:alptriso}
\tau^{-1}_{\bal}\wt\sN_k\simeq \wt\sN_k,\ee
where $pr_X\colon X\times \C^r\to X$ is the projection.
The standard basis on $\C^r$  induces an isomorphism \[\Z^r\simeq G, \quad \ba=(a_1,\dots,a_r)\mapsto g_\ba=\sum_i a_i t_i.\]
Then for $\ba\in\Z^r$, by construction we have 
\be\label{eq:idetraandgac}
\tau^{-1}_\ba\wt\sN_k=\wt{g_\ba\cdot \sN_k}
\ee
for every $k$.

%In particular, 
%\[\tau_{\bk}(\sN_0)=\sN_k\quad \textup{ and }\quad \tau_{\bl}\sM=\sM \textup{ for $\bl\in\Z^r$}.\]

\subsection{Relative maximal and minimal extension}\label{subsec:!*relex}
We now are ready to construct the relative maximal and minimal extensions.

We first recall a result of Sabbah \cite{Sab2} about the existence of generalized $b$-functions. 
\begin{theorem}[Sabbah]\label{thm:bssab}
Let $\cM$ be a holonomic $\shD_X$-module (with a fixed $\cM_0$). Then, locally on a relatively compact open subset $W\subseteq X$, there exists a polynomial $b(\bs)\in\C[\bs]$ such that the following two conditions are satisfied:
\begin{enumerate}
    \item $b(\bs)=\prod_{L\in \Z^r_{\ge0}}\prod_{\alpha\in \C}(L\cdot\bs+\alpha)^{n_{L,\alpha}}$ over some finite sets of $L$ and $\alpha$,
    \item $b(\bs)\cdot \frac{\sN_0}{\sN_{-1}}=0$.
\end{enumerate}
\end{theorem}
%\begin{proof}
%Since $\sN_k$ are lattices of $\sM$ by using graph embedding of $F$(cf. \cite[\S 5]{Wuch}), one can apply \cite[Theorem 4.23]{Wuch} and the proof is done.
%\end{proof}

\begin{prop}\label{prop:anstablej_*}
Let $\cM$ be a holonomic $\shD_X$-module. For every relatively compact open subset $W \subseteq X$ and for every relatively compact open subset $V\subseteq \C^r$, there exists an integer $l=l_{W,V}>0$ such that 
\begin{enumerate}
    %\item $g_1(V_\bal)\cap g_2(V_\bal)=\emptyset$ whenever $g_1\not= g_2\in G$, and
    \item ${\wt\sM}|_{W\times V}={\wt\sN_k}|_{W\times V}$ for all $k\ge l$.
    \item ${\wt\sN_{-k}}|_{W\times V}={\wt\sN_{-l}}|_{W\times V}$ for all $k\ge l$
\end{enumerate}
\end{prop}
\begin{proof}
By the construction of the analytic sheafification, we have 
\[b(\bs)\cdot \wt{\sN_0/\sN_{-1}}=0\]
where $b(\bs)$ is the generalized $b$-function in Theorem \ref{thm:bssab}. By translation, for every $k\in\Z$ with $\bk=(k,k,\dots,k)$, we have 
\be \label{eq:ufbfsab}
b(g_\bk\cdot\bs)\cdot \wt{\sN_{-k}/\sN_{-k+1}}=0,
\ee
Choose $l>0$ large enough (since $V$ is relatively compact) such that $b(g_\bk\cdot\bs)|_V$ are invertible for all $|k|>l$. Then by \eqref{eq:ufbfsab}, we have 
\[\wt\sN_k=\wt\sN_{k-1}\]
for all $|k|>l$. Part (2) then follows. Since $\displaystyle{\lim_{k\to\infty}} \sN_k=\sM$, part (1) also follows.
\end{proof}

The following theorem is a natural generalization of \cite[R\'esultat 1]{Mai} and \cite[Theorem 4.3.4]{BVWZ2}. Let us remark that one of the key steps of the proof is due to Maisonobe.
%However, a key step of the proof is due to Maisonobe. 
\begin{theorem}\label{thm:relholmaxexcc}
Let $\cM$ be a holonomic $\shD_X$-module. Then $\wt\sM$ and $\wt\sN_k$ are relative holonomic over the complex manifold $\C^r$ for all $k\in \Z$. Moreover, 
\[\CC^\rel(\wt\sM)=\CC^\rel(\wt\sN_k)=\CC(\cM(*D))\times \C^r.\]
\end{theorem}
\begin{proof}
By Proposition \ref{prop:anstablej_*} (1), $\sM$ is relative coherent over $\C^n$. Now we construct the relative characteristic cycles for $\wt\sN_k$ over a relatively compact open subset $W\subseteq X$. For simplicity we assume $X=W$.  
%For every $\bal\in\C^r$, substituting $\bs$ by $\bs+\bal$ gives an isomorphism \be\label{eq:isotau_a}
%\sN_k\simeq \tau_\bal(\sN_k).\ee
Then we define 
\[F_p(\sN_k)=(F_p\shD_X\otimes _\C \C[\bs])\cdot(\cM_0\cdot F^{\bs-\bk})\]
which gives a relative good filtration of $\sN_k$ over $\C[\bs]$ (cf. \cite[\S 3.1]{BVWZ}). Then $\wt{F_p(\sN_k)}$ gives a relative good filtration for $\wt\sN_k$ over $\C^r$ (cf. \cite[\S 2.1]{Wuch}). Using the isomorphism \eqref{eq:alptriso}, $\tau^{-1}_\bal\wt{F_p(\sN_k)}$ gives a relative good filtration of $\wt\sN_k$ such that 
\[\gr^{\tau^{-1}_\bal F}_\bullet (\wt\sN_k)\simeq\tau^{-1}_\bal\gr^F_\bullet \wt\sN_k,\]
%Thus, we also have 
%\[\gr^F_\bullet \tau_\bal(\wt\sN_k)\simeq\tau_\bal(\gr^F_\bullet \wt\sN_k).\]
Since characteristic varieties are independent of choices of good filtrations, this implies $\Ch^{\rel}(\wt\sN_k)$ is invariant under translation by arbitrary $\bal$. Thus, we have that $\Ch^{\rel}(\wt\sN_k)$ dominates $\C^r$ and $\Ch^{\rel}(\wt\sN_k)=\Lambda\times \C^k$ for some analytic conic subvariety $\Lambda\subseteq T^*X$. Since $\cM$ is holonomic, we can apply \cite[Proposition 13]{Mai} locally on $X$ and conclude that $\Lambda$ is Lagrangian. Therefore, $\wt\sN_k$ is relative holonomic over $\shD_{X\times \C^r/\C^r}$ for every $k$.  

We then pick $\ba=(a,a,\dots,a)\in \Z^r$ with $a\ll 0$. On a small neighborhood $V_\alpha$ of $\alpha$, by Proposition \ref{prop:anstablej_*} (1), we know 
\[\wt\sN_{k+a}|_{X\times V_\bal}=\wt\sM|_{X\times V_\bal}.\]
Thus, $\wt\sM|_{X\times V_\bal}$ is a relative holonomic $\shD_{X\times V_\bal/V_\bal}$-module.
Since $\sM$ is free over $\C[\bs]$, $\wt\sM|_{X\times (V_\bal)}$ is flat over $V_\bal$. Then by \cite[3.7.Lemme]{Sab2} we conclude that $\CC^\rel(\wt\sN_k)$ and $X\times\{\ba+\bal\}$ intersect properly and
\[\CC^\rel(\wt\sN_k)|_{X\times \{\ba+\bal\}}\simeq\CC^\rel(\wt\sN_{k+a})|_{X\times \{\bal\}}=\CC^\rel(\wt\sM)|_{X\times \{\bal\}}=\CC(\wt\sM|_{X\times \{\bal\}})\]
for every $\bal$. Since $\CC^\rel(\wt\sN_k)$ is invariant by arbitrary translation, $\CC^\rel(\wt\sN_k)$ and $X\times\{\bal\}$ intersect properly for every $\bal$ (i.e. $\CC^\rel(\wt\sN_k)$ is a fiberation over $\C^k$) and thus
\[\CC^{\rel}(\wt\sN_k)=\CC^{\rel}(\wt\sN_k)|_{X\times \{\bal\}}\times\C^r.\]
Taking $\bal=0$, since $\ba\in \C^n$ is a $\Z$-point,
\[\wt\sM|_{X\times \{\ba\}}\simeq \sM\otimes_{\C[\bs]}\C_{\ba}=\cM(*D)\]
where $\C_\ba$ is the residue field of $\ba\in \C^n$. We therefore obtain 
\[\CC^\rel(\wt\sN_k)=\CC(\cM(*D))\times\C^r\]
for every $k$.
Thanks to Proposition \ref{prop:anstablej_*} again, we have
\[\CC^\rel(\wt\sM)|_{X\times V_\bal}=\CC^\rel(\wt\sN_k)|_{X\times V_\bal}=\CC(\cM(*D))\times V_\bal \]
for $k\gg 0$.
Thus, globally 
\[\CC^\rel(\wt\sM)=\CC(\cM(*D))\times \C^r.\]
\end{proof}
\begin{remark}\label{rmk:necansh}
From the proof Proposition \ref{prop:anstablej_*}, one can see that it is necessary to use the analytic sheafification of $\sM$. To be more precise, even if $\cM$ is an algebraic holonomic $\shD_X$-module on a smooth complex algebraic variety $X$, $\wt\sM^{\textup{alg}}$ (cf. Remark \ref{rmk:algshd}) is not relative coherent over $\Spec R$, since every Zariski open subset of $\Spec R$ intersects the divisors $(b(g_\bk\cdot\bs)=0)$ for all but a finite number of $k\in \Z$, where $b(\bs)$ is the generalized $b$-function as in Theorem \ref{thm:bssab}. 
\end{remark}

We now write 
\[\wt\sM=\cM(*D^{(r)}_F),\]
calling it the \emph{relative maximal extension} of $\cM$ along $F$. Next we construct the minimal one. 
\begin{prop}\label{prop:j_*ncm}
If $\cM$ is holonomic, then $\wt\sM=\cM(*D^{(r)}_F)$ is $n$-Cohen-Macaulay.
\end{prop}
\begin{proof}
We fix a point $(x,p)\in X\times\C^r$. First, we know $(\wt\shD_{X,R})_{(x,p)}$ is Auslander regular (by for instance \cite[Theorem 4.3.2]{BVWZ}). Then by Auslander regularity, the grade number 
\[j(\Ext^k_{(\wt\shD_{X,R})_{(x,p)}}(\wt\sN_0, (\wt\shD_{X,R})_{(x,p)}))\ge k\]
for each $k$. Also, for all $k$ as right coherent $\wt\shD_{X,R}$-modules
\[\Ext^k_{\wt\shD_{X,R}}(\wt\sN_0, \wt\shD_{X,R})\]
are relative holonomic, as so is $\wt\sN$ by Theorem \ref{thm:relholmaxexcc}. By Lemma \ref{lm:angrnchreq}, we then know 
\[\Ext^k_{\wt\shD_{X,R}}(\wt\sN_0, \wt\shD_{X,R})=0 \textup{ for } k<n\]
and 
\[\dim_\C(\Ch^\rel(\Ext^k_{\wt\shD_{X,R}}(\wt\sN_0, \wt\shD_{X,R})))\le n+r-k,\textup{ for }n<k\le n+r.\]
By Lemma \ref{lm:altrelhol}, we conclude that $\wt\sN_0$ is $n$-Cohen-Macaulay away from $W\times V$ for $V\subsetneq \C^r$ some closed algebraic subvariety, where $W$ is a small neighborhood of $x\in X$. We now translate $\sN_0$ by $\tau_{-\ba}$ for some $\ba\in\Z^r_{\ge0}$ with each $a_i\gg 0$ so that $(x,p)\not\in \tau_{\ba}(W\times V)$. Thus, ${\tau^{-1}_{-\ba}\wt\sN_0}$ is Cohen-Macaulay around $(x,p)$. But from the proof of Proposition \ref{prop:anstablej_*}, such ${\tau^{-1}_{-\ba}\wt\sN_0}$ is just $\wt\sM$ around $(x,p)$.
%The proof is accomplished by Proposition\ref{prop:anstablej_*}(1). 
\end{proof}

\begin{remark}
\begin{enumerate}
\item After applying Theorem \ref{thm:relholmaxexcc}, Proposition \ref{prop:j_*ncm} also follows from \cite[Proposition 2]{FS}
\item In general, $\wt\sN_k$ are not $n$-Cohen-Macaulay. But since $\wt\sN_k\subseteq \wt\sM$ is a submodule of a $n$-Cohen-Macaulay module, $\wt\sN_k$ is always pure of codimension $n$ for every $k$ (by Proposition \ref{prop:suppsubcmrh}).
\end{enumerate}
\end{remark}

%Before we prove the above proposition, we need the following lemma.
%\begin{lemma}\label{lm:locrnm}
%If $\cM$ is holonomic, then $\sM_{x,m}=\sM_x\otimes_R R_m$ is $n$-Cohen-Macaulay over $\shD_{X,x}\otimes_\C R_m$ for all $x\in X$ and all maximal ideals $m\in \Spec R$ $($$R\simeq \C[\bs]$$)$.
%\end{lemma}
%\begin{proof}
%This lemma is the analytification of \cite[Theorem 5.2]{WZ}. The key input in proving Theorem 5.2 in \emph{loc. cite.} is the equality in \cite[Lemma 3.2.2(1)]{BVWZ} in the algebraic case. In the analytic local case this equality can be replaced by the equality in the first display on Page 17 in \cite[\S 3.6]{BVWZ}. More precisely, using the argument in the proof of \cite[Theorem 5.2]{WZ}, we have that for $k\gg0$, $$(\sN_k)_x\otimes_R R_m=\sM_{x,m}$$ are both $n$-Cohen-Macaulay over $\shD_{X,x}\otimes_\C R_m$.\end{proof}
%\begin{proof}[Proof of Proposition \ref{prop:j_*ncm}]
%By Lemma \ref{lm:stalkpre}, we know \[\wt\sM_{(x,p_m)}\simeq \sM_{x,m}\otimes_{\sO_{X,x}\otimes_\C R_m} \sO_{X\times\C^r,(x,p_m)}.\]Hence, \[\D(\wt\sM_{(x,p_m)})\simeq \D(\sM_{x,m})\otimes_{\sO_{X,x}\otimes_\C R_m}\sO_{X\times\C^r,(x,p_m)}.\]By passing to formal completions as in the proof of Lemma \ref{lm:suppan}, we have that \[\textup{$\wt\sM_{(x,p_m)}$ is $n$-Cohen-Macaulay $\Leftrightarrow$ $\sM_{(x,p_m)}$ is $n$-Cohen-Macaulay.}\]The proof is then done by Lemma \ref{lm:locrnm}.\end{proof}

We write by $\wt \sM_{\D\cM}$ the analytic sheafification of $(\D\cM)(*D)\otimes_{\C} \C[\bs]\cdot F^{-\bs}$. We apply Proposition \ref{prop:j_*ncm} for $\D\cM$, and thus see that $\wt\sM_{\D\cM}$ is $n$-Cohen-Macaulay and relative holonomic over $\C^r$. We then define 
\[\cM(!D_F^{(r)})=\D(\wt\sM_{\D\cM})\]
which is a relative holonomic $\shD$-module over $\C^r$, thanks to Cohen-Macaulayness. 
%We $\cM(!D_F^{(r)})$ the \emph{minimal extension along $F$}.
In particular, if $r=0$, then we have 
\[\cM(!D_F^{(0)})=\cM(!D)=\D((\D\cM)(*D)).\]
Notice that in the definition of  $\cM(!D_F^{(r)})$ we use the symbol $F^{-\bs}$ instead of $F^{\bs}$ because $\D$ naturally maps  $F^{-\bs}$ back to $F^{\bs}$ (see \cite[Lemma 5.3.1]{BVWZ2}).

Since $\cM(!D_F^{(r)})|_{(X\setminus D)\times\C^r}=\cM(*D_F^{(r)})|_{(X\setminus D)\times\C^r}$, the morphism $\sO_X\hookrightarrow \sO_X(*D)$ induces a natural morphism 
\[\cM(!D_F^{(r)})\longrightarrow \cM(*D_F^{(r)}).\]
\begin{lemma}\label{lm:!*incl}
For $\cM$ a holonomic $\shD_X$-module, if $r\ge 1$, then the natural morphism 
\[\cM(!D_F^{(r)})\hookrightarrow \cM(*D_F^{(r)})\]
is injective. %In particular, $\cM(!D_F^{(r)})$ is relative holonomic over $\C^r$.
\end{lemma}
\begin{proof}
We fix a small neighborhood $W\times V$ of a point $(x,p)\in X\times\C^r$. By construction, $\cM(*D_F^{(r)})$  has no non-zero coherent submodule supported  on $D\times\C^r$. By duality, %and Lemma \ref{lm:angrnchreq} 
$\cM(!D_F^{(r)})$  has no non-zero coherent quotient module supported  on $D\times\C^r$. Then the image of $\cM(!D_F^{(r)})\to \cM(*D_F^{(r)})$, denoted by $\cM(!*D_F^{(r)})$, is the minimal extension. Since the cokernel of $\cM(!D_F^{(r)})\to \cM(*D_F^{(r)})$ is supported on $D\times\C^r$, by analytic nullstellensatz, 
\[(\wt{\cM_0\cdot F^{\bs+\ba}})|_{W\times V}\subseteq \cM(!*D_F^{(r)})|_{W\times V}\]
for some $\ba\in\Z^r_{\ge 0}$ with each $a_i\gg0$, and thus
\[(\tau^{-1}_{\ba}\wt\sN_{0})|_{W\times V}\subseteq \cM(!*D_F^{(r)})|_{W\times V}\]
By minimality, 
\be\label{eq:miniex}
\cM(!*D_F^{(r)})|_{W\times V}=(\tau^{-1}_{\ba}\wt\sN_{0})|_{W\times V}\ee
(since otherwise $\frac{ \cM(!*D_F^{(r)})|_{W\times V}}{(\tau^{-1}_{\ba}\wt\sN_{0})|_{W\times V}}$ is a non-zero coherent quotient module of $\cM(!D_F^{(r)})$ supported  on $D\times\C^r$).
Similar to the proof of Proposition \ref{prop:j_*ncm}, $(\tau^{-1}_{\ba}\wt\sN_{0})|_{W\times V}$ is $n$-Cohen-Macaulay and hence so is $\cM(!*D_F^{(r)})|_{W\times V}$.

We use $\cK$ to denote the kernel on $W\times V$. If $\cK\not=0$, then we have a short exact sequence 
\[0\to \cK\to \cM(!D_F^{(r)})|_{W\times V}\to \cM(!*D_F^{(r)})|_{W\times V}\to0.\]
Since $\cK$ is relative holonomic, by Proposition \ref{prop:suppsubcmrh} and Proposition \ref{lm:altrelhol}, 
\[\Ch^\rel(\cK)=\Lambda_\cK\times\C^r\]
for some Lagrangian subvariety $\Lambda_\cK\subseteq T^*X|_W$. Taking the dual, we have 
\[0\to (\D\cM)(!*D_F^{(r)})|_{W\times V}\to (\D\cM)(*D_F^{(r)})|_{W\times V}\to\D(\cK)\to0.\]
By \eqref{eq:miniex} (replacing $\cM$ by $\D\cM$) and Theorem \ref{thm:relholmaxexcc},
\[\CC((\D\cM)(!*D_F^{(r)})|_{W\times V})=\CC((\D\cM)(*D_F^{(r)})|_{W\times V}).\]
Since $\Ch^\rel(\D(\cK))=\Lambda_\cK\times\C^r$, we get a contradiction by counting multiplicities. 
\end{proof}

An immediate consequence of the inclusion in the lemma above is that $\cM(!D_F^{(r)})$ is the \emph{minimal extension} (analogous to the Deligne-Goresky-MacPherson extensions for perverse sheaves \cite[\S 8.2.1]{HTT}). Therefore, when $\cM$ is holonomic, $\cM(!D_F^{(r)})$ satisfies the property that it has no non-trivial submodule (or quotient module) supported on $D\times \C^r$ in the category of relative coherent $\wt\shD_{X,R}$-modules. By this property, in the situation of Proposition \ref{prop:anstablej_*},  we have 
\be\label{eq:locj!st}
\cM(!D_F^{(r)})|_{W\times V}=\wt\sN_{-k}|_{W\times V}
\ee
for all $k\ge l_{W\times V}$.

\subsection{Gluing data and $G$-equivariance}\label{subsec:gluemmex}
Now we discuss how to ``glue"  the local data in Proposition \ref{prop:anstablej_*} to construct the maximal and minimal extensions alternatively.

We keep assuming $\cM$ a holonomic $\shD_X$-module and $W$ a relatively compact open subset of $X$. We pick a locally finite covering 
\[\C^r=\bigcup_\beta V_\beta\]
such that each $V_\beta$ is relatively compact. For instance, we can pick a finite covering for the unit polydisc in $\C^r$ and then translate the covering by $\Z^r$ to get such a covering of $\C^r$. We can tautologically ``glue" $\cM(!D^{(r)}_F)$ with the help of Proposition \ref{prop:anstablej_*} (2) as follows: 
\begin{enumerate}[label=(\roman*)]
    \item For every $V_{\beta_0}$, we have 
    \[\cM(!D^{(r)}_F)|_{W\times V_{\beta_0}}=\wt\sN_{-l_{W,V_{\beta_0}}}|_{W\times V_{\beta_0}}\]
    \item We set \[k_0=\max\{l_{W,V_{\beta_i}}\mid V_{\beta_0}\cap V_{\beta_i}\not=\emptyset\}.\]
    \item If $V_{\beta_0}\cap V_{\beta_i}\not=\emptyset$, we can ``glue" $\wt\sN_{-l_{W,V_{\beta_0}}}|_{W\times V_{\beta_0}}$ and $\wt\sN_{-l_{W,V_{\beta_i}}}|_{W\times V_{\beta_i}}$ since 
    \[\wt\sN_{-l_{W,V_{\beta_0}}}|_{W\times (V_{\beta_0}\cap V_{\beta_i})}=\wt\sN_{-k_0}|_{W\times (V_{\beta_0}\cap V_{\beta_i})}= \wt\sN_{-l_{W,V_{\beta_i}}}|_{W\times (V_{\beta_0}\cap V_{\beta_i})}.\]
\end{enumerate}
Notice that the cocycle condition of ``gluing" is guaranteed by Condition (ii). Using Proposition \ref{prop:anstablej_*} (1), we can ``glue" $\cM(*D^{(r)}_F)$ similarly.
\begin{theorem}\label{thm:!*Gequiv}
Let $\cM$ be a holonomic $\shD_X$-module. Then both $\cM(*D^{(r)}_F)$ and $\cM(!D^{(r)}_F)$ are $G$-equivariant sheaves of abelian groups on $X\times \C^r$ and the inclusion
\[\cM(!D^{(r)}_F)\hookrightarrow\cM(*D^{(r)}_F)\]
is a $G$-equivariant morphism (see \cite[Part I.0.2]{BL} for definitions). 
\end{theorem}
\begin{proof}
By \eqref{eq:alptriso} and \eqref{eq:idetraandgac}, we have for every $g\in G$
\[\wt{g\cdot\sN_k}\simeq \wt\sN_k\]
In particular,
\be\label{eq:anshgact}
\wt{g\cdot\sN_k}|_{W\times V_{\beta_0}}\simeq \wt\sN_k|_{W\times V_{\beta_0}}.
\ee
The isomorphism above is obviously compatible with the ``gluing" data for both $\cM(*D^{(r)}_F)$ and $\cM(!D^{(r)}_F)$. 
Since $G$ acts on $X\times \C^r$, for every $g\in G$ we have an isomorphism 
\[g: X\times \C^r\to X\times \C^r.\]
The ``gluing" data and \eqref{eq:anshgact} together give isomorphisms 
\be\label{eq:geqvassab}
g^{-1}(\cM(*D^{(r)}_F))\simeq\cM(*D^{(r)}_F)\textup{ and } g^{-1}(\cM(!D^{(r)}_F))\simeq\cM(!D^{(r)}_F).
\ee
Since $G$ is a discrete group, by \cite[Part I.0.2.Remark]{BL} the isomorphisms above make  $\cM(*D^{(r)}_F)$ and $\cM(!D^{(r)}_F)$ $G$-equivariant sheaves of abelian groups (one can easily check the cocycle condition). 

%In fact, the isomorphism $g^*(\cM(*D^{(r)}_F))=\cM(*D^{(r)}_F)$ is also directly induced by the identity $g\cdot \sM=\sM$. 
The inclusion 
$\cM(!D^{(r)}_F)\hookrightarrow\cM(*D^{(r)}_F)$ is $G$-equivariant because of the following commutative diagram
\[\begin{tikzcd}
 \sN_l\arrow[r,"\simeq"]\arrow[d,hook] & g\cdot \sN_l\arrow[d,hook] \\
  \sN_k\arrow[r,"\simeq"] & g\cdot \sN_k)
\end{tikzcd}
\]
for all $l<k$.
\end{proof}
For $\cM$ a holonomic $\shD_X$-module, we now define 
\[\Psi_F(\cM)=\frac{\cM(*D^{(r)}_F)}{\cM(!D^{(r)}_F)}.\]
By Lemma \ref{lm:!*incl}, another interpretation of $\Psi_F(\cM)$ is 
\[\Psi_F(\cM)[-1]\stackrel{q.i.}{\simeq} R\Gamma_{[D\times \C^r]}(\cM(!D^{(r)}_F)).\]
%where the latter is the derived algebraic local cohomology sheaf (cf. \cite[II.5]{Bj}).

\begin{theorem}\label{thm:gnbGeq}
Let $\cM$ be a holonomic $\shD_X$-module. Then $\Psi_F(\cM)$ is $G$-equivariant and relative holonomic over $\C^r$. If moreover, $\Psi_F(\cM)\not=0$, then it is $(n+1)$-Cohen-Macaulay.
\end{theorem}
\begin{proof}
As a quotient of $\cM(*D_F^{(r)})$, $\Psi_F(\cM)$ is relative holonomic. Since the category of $G$-equivariant sheaves is abelian and the inclusion 
\[\cM(!D^{(r)}_F)\hookrightarrow\cM(*D^{(r)}_F)\]
is $G$-equivariant by Theorem \ref{thm:!*Gequiv}, $\Psi_F(\cM)$ is $G$-equivariant.

To prove Cohen-Macaulayness, we consider the short exact sequence
\[0\to \cM(!D^{(r)}_F)\rightarrow\cM(*D^{(r)}_F)\rightarrow \Psi_F(\cM)\to 0.\]
By Theorem \ref{thm:relholmaxexcc} and \eqref{eq:locj!st}, we know 
\[\CC^\rel(\cM(!D^{(r)}_F))=\CC^\rel(\cM(*D^{(r)}_F)).\]
If $\Psi_F(\cM)\not=0$, by counting multiplicity we have \[\dim_\C\Ch^\rel(\Psi_F(\cM))\le n+r-1.\]
By Lemma \ref{lm:angrnchreq}, 
\[\Ext^k_{\wt\shD_{X,R}}(\Psi_F(\cM))=0 \textup{ for $k\le n$}.\]
Taking the dual of the short exact sequence, since both $\cM(!D^{(r)}_F)$ and $\cM(*D^{(r)}_F)$ are $n$-Cohen-Macaulay, $\Psi_F(\cM)$ is $(n+1)$-Cohen-Macaulay.
\end{proof}

\begin{remark}\label{rmk:reldsteqdb}
(1) By Theorem \ref{thm:!*Gequiv} and Theorem \ref{thm:gnbGeq}, $\cM(*D_F^{(r)})$, $\cM(!D_F^{(r)})$ and $\Psi_F(\cM)$ are $G$-equivariant sheaves of abelian groups. Since  the isomorphisms in \eqref{eq:geqvassab} are not $\sO_{X\times\C^r}$-linear (as \eqref{eq:subtrop} is not $\C[\bs]$-linear), $\cM(*D_F^{(r)})$, $\cM(!D_F^{(r)})$ and $\Psi_F(\cM)$ are not $G$-equivariant $\sO_{X\times\C^r}$-modules or $\wt\shD_{X,R}$-modules. However, since $\pi$ is covering, the sheaf morphism
\[\pi^{-1}(\shD_{X\times(\C^*)^r/(\C^*)^r})\rightarrow \wt\shD_{X,R}\]
is an isomorphism locally and hence also a global isomorphism. In particular, $\wt\shD_{X,R}$ is G-equivariant (see \cite[Part I.0.3.Lemma]{BL}) and its $G$-equivariance is induced from the deck transformation of the universal covering of $(\C^*)^r$. Equivalently, the $G$-equivariance of $\wt\shD_{X,R}$ is induced from the $G$-action on $\C[\bs]$ given by \eqref{eq:gactiononR}. Moreover, the $G$-equivariance of $\cM(*D_F^{(r)})$, $\cM(!D_F^{(r)})$ and $\Psi_F(\cM)$ is also induced from the $G$-action on $\C[\bs]$. Then, for $g\in G$, $Q\in \pi_*(\wt\shD_{X,R})$  and $m\in \pi_*(\cM(*D_F^{(r)}))$ (similarly for $\cM(!D_F^{(r)})$ and $\Psi_F(\cM)$), we have 
\[g\cdot (Q\cdot m)=(g\cdot Q)\cdot (g\cdot m).\]
Therefore, the $G$-invariant parts $\pi^G_*(\cM(*D_F^{(r)}))$ $\pi_*^G(\cM(!D_F^{(r)}))$ and $\pi_*^G(\Psi_F(\cM))$ are $\pi^G_*(\wt\shD_{X,R})\simeq \shD_{X\times(\C^*)^r/(\C^*)^r}$-modules. Since $\pi$ is covering, they are automatically relative holonomic $\shD_{X\times(\C^*)^r/(\C^*)^r}$-modules.

(2) Since
\be\label{eq:dualminmax}
\D(\cM(!D^{(r)}_F))\simeq (\D\cM)(*D^{(r)}_F) \textup{ and } \D(\cM(*D^{(r)}_F))\simeq (\D\cM)(!D^{(r)}_F),
\ee
$\D(\cM(!D^{(r)}_F))$ (resp. $\D(\cM(*D^{(r)}_F))$) is also $G$-equivariant. Since $\pi^{-1}\circ \pi^G_*\simeq \id$ thanks to \cite[Part I.0.3.Lemma]{BL} again, we have 
\be\label{eq:dualequivdirct}
\pi_*^G \D\big(\cM(!D^{(r)}_F)\big)\simeq \D \pi_*^G\big(\cM(!D^{(r)}_F)\big)
\ee 
and
\be\label{eq:dualequivdirct1}
 \pi_*^G \D\big(\cM(*D^{(r)}_F)\big)\simeq \D \pi_*^G\big(\cM(*D^{(r)}_F)\big).
\ee 

\end{remark}

Theorem \ref{thm:gnbGeq} and \eqref{eq:dualminmax} together imply:
\begin{coro}\label{cor:dualalcop}
For $\cM$ a holonomic $\shD_X$-module, we have a quasi-isomorphism 
\[\D(\Psi_F(\cM))\stackrel{q.i.}{\simeq}\Psi_F(\D\cM)[-1].\]
\end{coro}

Now, we discuss relative intermediate extensions and use them to study local cohomology sheaves for $\cM(!D_F^{(r)})$.
For every subset $I\subseteq \{1,2,\dots,r\}$, we denote
\[D_I=\bigcap_{i\in I} D_i\textup{ and } D^I=\bigcup_{i\in I} D_i\]
and 
\[\wt D_I=D_I\times \C^r\textup{ and } \wt D^I=D^I\times \C^r.\]
Then we define 
\[\cM(!D_F^{(r)})(*\wt D^I)\coloneqq \lim_{k\to +\infty} \cM(!D_F^{(r)})\otimes_\sO \sO_{X\times\C^r}(k\wt D^I).\]
If $J\subseteq I$, then similar to Theorem \ref{thm:!*Gequiv} we have a $G$-equivariant inclusion
\be\label{eq:eqvintex1}
\cM(!D_F^{(r)})(*\wt D^J)\hookrightarrow \cM(!D_F^{(r)})(*\wt D^I).
\ee
In fact, the inclusion in Theorem \ref{thm:!*Gequiv} factor through \eqref{eq:eqvintex1} when $J$ is empty. 
Similar to Proposition \ref{prop:j_*ncm} (see also \cite[\S 5.3]{BVWZ2} for the algebraic case), $$\cM(!D_F^{(r)})(*\wt D^I)$$
 is $n$-Cohen-Macaulay. We thus define
 \[\cM(*D_F^{(r)})(!\wt D^I)\coloneqq \D\big((\D\cM)(!D_F^{(r)})(*\wt D^I) \big).\]
 If $J\subseteq I$, then similar to \eqref{eq:eqvintex1} we have a $G$-equivariant inclusion
\be\label{eq:eqvintex2}
\cM(*D_F^{(r)})(!\wt D^I)\hookrightarrow \cM(*D_F^{(r)})(!\wt D^J).
\ee
Furthermore, similar to \eqref{eq:dualequivdirct} in Remark \ref{rmk:reldsteqdb} (2), we have 
\be\label{eq:dualeqvdrig}
\pi_*^G\D\big(\cM(!D_F^{(r)})(*\wt D^J)\big)\simeq \D\pi_*^G\big(\cM(!D_F^{(r)})(*\wt D^J)\big).
\ee
By \eqref{eq:eqvintex1} and \eqref{eq:eqvintex2}, for every subset $I$, $\cM(!D_F^{(r)})(*\wt D^I)$ and $\cM(*D_F^{(r)})(!\wt D^I)$ are \emph{relative intermediate extensions} lying in between $\cM(!D_F^{(r)})$ and $\cM(*D_F^{(r)})$. 
\begin{theorem}\label{thm:loccohsheq}
Let $\cM$ be a holonomic $\shD_X$-module. Then for every $I\subseteq \{1,2,\dots,r\}$, the complexes $R\Gamma_{[\wt D_I]}\cM(!D_F^{(r)})$ and $\D R\Gamma_{[\wt D_I]}\big((\D\cM)(!D_F^{(r)})\big)$ are $G$-equivariant. \\Moreover, 
$$\pi_*^GR\Gamma_{[\wt D_I]}\big(\cM(!D_F^{(r)})\big)\simeq R\Gamma_{[D_I\times (\C^*)^r]}\pi_*^G\big(\cM(!D_F^{(r)})\big)$$
and 
$$\pi_*^G\D R\Gamma_{[\wt D_I]}\big((\D\cM)(!D_F^{(r)})\big)\simeq \D R\Gamma_{[D_I\times (\C^*)^r]}\pi_*^G\big((\D\cM)(!D_F^{(r)})\big)$$
are complexes of relative holonomic $\shD_{X\times(\C^*)^r/(\C^*)^r}$-modules.
\end{theorem}
\begin{proof}
By the definition of $R\Gamma_{[\wt D_I]}$, $R\Gamma_{[\wt D_I]}\cM(!D_F^{(r)})$ is represented by the Koszul-type complex 
\[
\begin{aligned}
0\to\cM(!D_F^{(r)}) \rightarrow \cdots &\rightarrow \bigoplus _{J\subseteq I, |J|=i}\cM(!D_F^{(r)})(*\wt D^J)\rightarrow \\
&\rightarrow\bigoplus _{J\subseteq I, |J|=i+1}\cM(!D_F^{(r)})(*\wt D^J) \rightarrow \cdots \cM(!D_F^{(r)})(*\wt D^I)\to 0.
\end{aligned}
\]
with cohomogy degrees $0,1,2, \dots |I|$,
and taking duality of the above complex for $\D\cM$, $\D R\Gamma_{[\wt D_I]}\big((\D\cM)(!D_F^{(r)})\big)$ is the complex
\[
\begin{aligned}
0\to\cM(*D_F^{(r)})(!\wt D^I) \rightarrow \cdots &\rightarrow \bigoplus _{J\subseteq I, |J|=i}\cM(*D_F^{(r)})(!\wt D^J)\rightarrow \\
&\rightarrow\bigoplus _{J\subseteq I, |J|=i-1}\cM(*D_F^{(r)})(!\wt D^J) \rightarrow \cdots \cM(*D_F^{(r)})\to 0.
\end{aligned}
\]
with cohomogy degrees $-|I|,-|I|+1,\dots, -1,0$. By using \eqref{eq:eqvintex1} and \eqref{eq:eqvintex2}, the two complexes as above are $G$-equivariant. The second required claim follows from Remark \ref{rmk:reldsteqdb} and \eqref{eq:dualeqvdrig}.
\end{proof}

\subsection{$G$-equivariance of relative de Rham complex}\label{subsec:deRhamequiv}
Let $\sE$ be $G$-equivariant sheaf and a $\shD_{X\times\C^r/\C^r}$-module (see Remark \ref{rmk:reldsteqdb}).  For a differential operator $P\in \shD_X$, since the $G$-action on $\sE$ is induced from the $G$-action on $\C^r$, we have 
\[g^*(P\cdot \sE)\simeq P\cdot(g^*\sE),\]
which means that the $P$-action on $\sE$ is $G$-equivariant. 

Now we fix a local coordinates $(x_1,\dots,x_n)$ of $X$ and denote by $\partial_{x_i}$ the vector field of $x_i$. Then locally the relative de Rham complex of $\sE$ satisfying 
\be\label{eq:relderhamkos}
\DR_{X\times\C^r/\C^r}(\sE)\simeq \Kos(\sE; \partial_{x_1},\dots,\partial_{x_r}),
\ee
where $\Kos(\sE; \partial_{x_1},\dots,\partial_{x_r})$ is the Koszul complex of the $\partial_{x_i}$-actions on $\sE$. Since $G$ acts on the first factor of $X\times \C^r$ trivially, $\partial_{x_i}$-actions on $\sE$ are $G$-equivariant (see \cite[Part I.0.2.]{BL} for the definition of equivariant morphisms). 
We thus obtain:
\begin{lemma}\label{lm:DReq}
Let $\sE$ be a sheaf (or complexes) of $\shD_{X\times\C^r/\C^r})$-modules and $G$-equivariant. Then the de Rham complex $\DR_{X\times\C^r/\C^r}(\sE)$ is $G$-equivariant and hence we have a natural isomorphism
\[\pi_*^G\big(\DR_{X\times\C^r/\C^r}(\sE)\big)\simeq \DR_{X\times(\C^*)^r/(\C^*)^r}(\pi_*^G\sE).\]
\end{lemma}

%\begin{remark}\label{rmk:dualGeqv}
%If $\D(\sE)$ is also $G$-equivariant, one can see in a similar way that we have a natural isomorphism 
%\[\pi_*^G\D(\sE)\simeq \D\pi_*^G(\sE).\]
%\end{remark}

\subsection{Bernstein-Sato ideal along $F$}\label{subsec:BSideal}

For every $\ba=(a_1,\dots,a_r)\in\Z^r_{\ge 0}$, the Bernstein-Sato ideal of $\sN_0$ and $\ba$ along $F$ is the ideal
\[B_F(\sN_0,\ba)=B(\frac{\sN_0}{g_\ba\cdot\sN_0})\subseteq \C[\bs].\]
For simplicity, we write $B_F(\sN_0)=B_F(\sN_0,{\mathbf1}_r)$ with ${\mathbf 1}_r=(1,1,\dots,1)\in \Z^r$.
When $\cM_0=\sO_X$, we further denote $B_F(\sN_0)$ by $B_F$ for short. Notice that $B_F$ is the major Bernstein-Sato ideal studied in the literature; see for instance \cite{Budur,BVWZ, Maihyp}.

Bernstein-Sato ideals along $F$ are natural generalization of the Bernstein-Sato polynomial for a single holomorphic function (see \cite{KasBf} and \cite[\S 4.1]{Budur}). 
By Theorem \ref{thm:bssab} (replacing $\sN_{-1}$ by $g_\ba\cdot\sN_0$), we have that, if $\cM$ is holonomic, then locally on a relatively compact open subset $W\subseteq X$
\[\textup{$B_F(\sN_0,\ba)\not=0$ for all $\ba\in\Z^r_{\ge 0}$ and $\ba\not=\mathbf 0$.}\]

The following theorem is essentially due to Maisonobe \cite{Mai}; here we present it in a more generalized form.\footnote{In \emph{loc. cit.} the author only considered the cyclic modules $\shD_X[\bs]m\cdot F^\bs$ for sections $m$ of $\cM$ (instead of $\sN_0$).}
See also \cite{BVWZ} for $\cM_0=\sO_X$ under the algebraic setting.

\begin{theorem}
Let $\cM$ be a regular holonomic $\shD_X$-module. Locally on a relatively compact open subset $W\subseteq X$, if $\sN_0/\sN_{-1}$ is not zero on $W$ $(\Leftrightarrow W\cap D\not=\emptyset)$, then 
$Z(B_F(\sN_0))\subseteq \C^r$ is of codimension-one $($but not purely in general $)$. 
\end{theorem}
\begin{proof}
If $W\cap D=\emptyset$, then obviously $\sN_0=\sN_{-1}$. We now assume $W\cap D\not=0$ and then prove $\sN_0/\sN_{-1}\not=0$ on $W$. 

To this purpose, we need to treat $\sN_k$ as a logarithmic $\shD$-module. Using the graph embedding (see \cite[\S 5]{Wuch}), the inclusion
\[\iota_{F,*}(\sN_0)\hookrightarrow \iota_{F,*}(\sM)\simeq \iota_{F,+}(\cM(*D))\]
makes $\iota_{F,*}(\sN_0)$ a lattice of the regular holonomic $\shD_Y$-module $\iota_{F,+}(\cM(*D))$, where 
\[\iota_{F}\colon X\hookrightarrow Y=X\times \C^r_{\mathbf u}, \quad x\mapsto (x,f_1(x),\dots,f_r(x))\]
is the graph embedding with ${\mathbf u}=(u_1,\dots,u_r)$ the coordinates
and $\iota_{F,+}$ denotes the $\shD$-module pushforward functor of $\iota_F$.

Since $\iota_F$ is a closed embedding, we have
\[\frac{\sN_0}{\sN_{-1}}\not=0 \Leftrightarrow \frac{\iota_{F,*}(\sN_0)}{\iota_{F,*}(\sN_{-1})}\not=0.\]
By the construction of $\iota_{F,+}$, we have 
\[\prod_{i=1}^r u_i\cdot\iota_{F,*}(\sN_0)=\iota_{F,*}(\sN_{-1}).\]
Therefore, we have a short exact sequence 
\[0\to \iota_{F,*}(\sN_0)\xrightarrow{\prod_{i=1}^r u_i}\iota_{F,*}(\sN_0)\longrightarrow \frac{\iota_{F,*}(\sN_0)}{\iota_{F,*}(\sN_{-1})}\to0.\]
For simplicity, we write $$\sG=\iota_{F,*}(\sN_0),\quad \sH=\frac{\iota_{F,*}(\sN_0)}{\iota_{F,*}(\sN_{-1})} \quad \textup{ and } \quad v=\prod_{i=1}^r u_i.$$
Since $\sG$ is a lattice along the normal crossing divisor $D_Y=(\prod_{i=1}^r u_i=0)$, by \cite[Theorem 1.6]{Wuch} the logarithmic characteristic variety satisfies
\[\Ch^{\log}(\sG))=\overline{\wt\iota_{F}(\Ch(\cM|_{X\setminus D}))},\]
where $\wt\iota_{F}$ is the induced embedding
\[\wt\iota_{F}\colon T^*(X\setminus D)\hookrightarrow T^*(Y\setminus D_Y)\simeq T^*(X\setminus D)\times T^*\C^r_{\mathbf u}, \quad (x,\xi)\mapsto (x,\xi,F(x),dF(x))\]
and the closure is taken inside the logarithmic cotangent bundle $T^*(Y,D_Y)$. In particular, $\Ch^{\log}(\sG)$ has no component over $(v=0)$.

We now take a good filtration $F_\bullet\sG$ over $F_\bullet\shD_{Y,D_Y}$ (cf. \cite[\S 3]{Wuch}). We thus have a filtered morphism 
\[F_\bullet\sG\xrightarrow{v} F_\bullet\sG\]
with the associated graded morphism
\[\varphi\colon\gr^{F}_\bullet\sG\xrightarrow{v} \gr^{F}_\bullet\sG.\]
Since the filtration is bounded from below, the filtered morphism gives a convergent spectral sequence. By convergence (see for instance \cite[Lemma 3.5.13 (iii)]{Laumon}), we have 
\[[\gr^F_\bullet\sH]=[\Coker\varphi]-[\ker\varphi]\]
in the Grothendieck group $K_0$ generated by $\gr^F_\bullet\shD_{Y,D_Y}$-modules with support of dimension $\le n+r-1$, where $F_\bullet\sH$ is the induced filtration on the quotient. But one can check in $K_0$
\[[\Coker\varphi]-[\ker\varphi]=[(\gr^F_\bullet\sG/\cT)|_{(v=0)}]\]
where $\cT\subseteq \gr^F_\bullet\sG$ is the $v$-torsion subsheaf. Since $\gr^F_\bullet\sG/\cT$ has no $v$-torsion, we have 
\[\supp_{\gr_\bullet^F\shD_{Y,D_Y}}((\gr^F_\bullet\sG/\cT)|_{(v=0)})=\supp_{\gr_\bullet^F\shD_{Y,D_Y}}(\gr^F_\bullet\sG/\cT)|_{(v=0)}.\]
Since 
\[\Ch^{\log}(\sG)=\supp_{\gr_\bullet^F\shD_{Y,D_Y}}\gr^F_\bullet\sG\]
and $\Ch^{\log}(\sG)$ does not have component over $(v=0)$, we have 
\[\supp_{\gr_\bullet^F\shD_{Y,D_Y}}\gr^F_\bullet\sG=\supp_{\gr_\bullet^F\shD_{Y,D_Y}}(\gr^F_\bullet\sG/\cT).\]
By \cite[Proposition 1.5.1]{Gil}, we conclude 
\[\Ch^{\log}_{n+r-1}(\sH)=\Ch^{\log}(\sG)|_{v=0}\]
where $\Ch^{\log}_{n+r-1}(\sH)$ is the purely $(n+r-1)$-dimensional part of the characteristic variety. But by construction,
\[\Ch^{\log}(\sH)\subseteq\Ch^{\log}(\sG)|_{v=0}\]
and hence 
\[\Ch^{\log}(\sH)=\Ch^{\log}(\sG)|_{v=0}.\]
In fact, the above argument proves more generally
 \[\CC^{\log}(\sH)=\CC^{\log}(\sG)|_{v=0}.\]
In particular, $\sH\not=0$ and hence so is $\sN_0/\sN_{-1}$.
Then $\sN_k/\sN_{-k}\not=0$ for all $k>0$.
Since for $k\gg0$ $\wt\sN_k/\wt\sN_{-k}$ is just $\Psi_F(\cM)$ analytically locally over $\C^r$, we now conclude that $\Psi_F(\cM)$ is $(n+1)$-Cohen-Macaulay by Theorem \ref{thm:gnbGeq}. By Lemma \ref{lm:angrnchreq}, this also implies 
\[\dim_\C(\Ch^\rel(\wt\sN_k/\wt\sN_{-k}))=n+r-1.\]
Thus, by \cite[Lemma 2.3]{Wuch}
\[\dim_\C(\Ch^\rel(\wt\sN_l/\wt\sN_{l-1}))=n+r-1\]
for some $-k<l\le k$. But by translation,
\[\sN_l/\sN_{l-1}\simeq \sN_0/\sN_{-1}\]
and hence 
\[\dim_\C(\Ch^\rel(\wt\sN_0/\wt\sN_{-1}))=n+r-1.\]
The proof is then accomplished by Lemma \ref{lm:altrelhol} and Proposition \ref{prop:zbfsupp=}. 
\end{proof}

The theorem above and Theorem \ref{thm:bssab} together implies:
\begin{coro}\label{cor:regholzbfnontr}
Let $\cM$ be a regular holonomic $\shD_X$-module. Locally on a relatively compact open subset $W\subseteq X$ intersecting $D$, $Z_{r-1}(B_F(\sN_0))$ is a finite union of translated hyperplanes of the form $(L\cdot\bs+\alpha=0)$
with $L\in\Z^r_{\ge0}$ and $\alpha\in\C$,
where $Z_{r-1}(B_F(\sN_0))$ denotes the pure $(r-1)$-dimensional part of $Z(B_F(\sN_0))$.
\end{coro}
The corollary above enables us to make the following definition:
\begin{definition}\label{def:slopkappa}
For $\cM$ a regular holonomic $\shD_X$-module, locally on a relatively compact open subset $W\subseteq X$, we define
\[S(F,\cM)=\{\textup{primitive }L\in \Z^r_{\ge 0}\mid (L\cdot \bs+\alpha=0)\subseteq Z_{r-1}(B_F(\sN_0)),\textup{ for some }\alpha\in\C\}\]
and that for $L\in S(F,\cM)$
\[\kappa(L)=\{\alpha\in \C\mid (L\cdot \bs+\alpha=0)\subseteq Z_{r-1}(B_F(\sN_0))\}.\]
We set 
\[\wt\kappa(L)=\kappa(L)/\sim_L\]
where $\sim_L$ denotes the equivalence 
\[\alpha_1\sim_L \alpha_2 \Longleftrightarrow \alpha_2=L\cdot \ba+\alpha_1 \textup{ for some } \ba\in\Z^r.\]
%If $W\cap D=\emptyset$, we set $S(F,\cM)=\emptyset$ by convention.
\end{definition}
Since $\sN_0$ is depending on $\cM_0$, $B_F(\sN_0)$ is depending on $\cM_0$. However, one can easily check that
\[\textup{$S(F,\cM)$ is independent of the choices of $\cM_0$.}\] See \cite[R\'esultat 6]{Mai} for a logarithmic interpretation of $S(F,\cM)$ and also \cite[Proposition 4.4.4]{BVWZ2}.

\begin{proof}[Proof of Theorem \ref{thm:relsd}]
For $p\in Z_{r-1}(B_F(\sN_0))$ and for every $g\in G$, by Proposition \ref{prop:zbfsupp=} we know \[(\frac{g\cdot\sN_0}{g\cdot\sN_{-1}})_{m_{g\cdot p}}\not=0,\]
where $m_{g\cdot p}$ is the maximal ideal of $p$. We take $k\gg 0$ such that $g\cdot\sN_0/g\cdot\sN_{-1}$ is a subquotient of $\sN_k/\sN_{-k}$. Then 
\[(\frac{\sN_k}{\sN_{-k}})_{m_{g\cdot p}}\not=0.\]
By Lemma \ref{lm:suppan}, we then have for some $x\in X$
\[(\frac{\wt\sN_k}{\wt\sN_{-k}})_{(x,g\cdot p)}\not= 0.\]
But by Proposition \ref{prop:anstablej_*} (1) and \eqref{eq:locj!st}, $\wt\sN_k/\wt\sN_{-k}$ is $\Psi_F(\cM)$ locally around $(x,g\cdot p)$. We thus have 
\[\bigcup_{g\in G}\bigcup_{L\in S(F,\cM)}\bigcup_{\alpha\in \kappa(L)}(g\cdot(L\cdot\bs+\alpha)=0)\subseteq\supp_{\C^r}\Psi_F(\cM).\]

Conversely, we assume $\Psi_F(\cM)_{(x,p)}\not=0$ for some $x\in X$ and for $p\in \C^r$ away from 
\[\bigcup_{g\in G}\bigcup_{L\in S(F,\cM)}\bigcup_{\alpha\in \kappa(L)}(g\cdot(L\cdot\bs+\alpha)=0).\]
Then for some $k\gg 0$
$$\Psi_F(\cM)|_{W\times V}=\frac{\wt\sN_k}{\wt\sN_{-k}}|_{W\times V}\not=0$$ with $W$ a small neighborhood of $x$ and $V$ a small neighborhood of $p$. Since $\Psi_F(\cM)$ is $(n+1)$-Cohen-Macaulay, by Lemma \ref{lm:angrnchreq} and Proposition \ref{prop:suppsubcmrh}, we know \[\supp_{\C^r}\Psi_F(\cM)|_{W\times V}\]
is purely of codimension one. But by Theorem \ref{thm:bssab} and Proposition \ref{prop:zbfsupp=}, we know 
\[\supp_{\C^r}\frac{\sN_k}{\sN_{-k}}\]
is contained in a union of translated linear hyperplanes in $\C^r$ of forms $(L\cdot\bs+\alpha=0)$ with $L\in \Z^r_{\ge 0}$. Therefore, we can assume $ (L\cdot\bs+\alpha=0)\cap V \ni p$ is a component of $\supp_{\C^r}\Psi_F(\cM)|_{W\times V}$. Then $(L\cdot\bs+\alpha=0)$ must be a component of $\supp_{\C^r}\sN_{l}/\sN_{l-1}$ for some $-k<l\le k$. But this means 
$$p\in \bigcup_{g\in G}\bigcup_{L\in S(F,\cM)}\bigcup_{\alpha\in \kappa(L)}(g\cdot(L\cdot\bs+\alpha)=0),$$
which is a contradiction.
\end{proof}

\subsection{Evaluating $\Psi_F(\cM)$}
In this subsection, we discuss ``evaluating" $\cM(*D_F^{(r)})$, $\cM(!D_F^{(r)})$ and $\Psi_F(\cM)$ at points in $\C^r$ and applying them to prove Theorem \ref{thm:relccgny}. We first introduce some notations.

For $\bal=(\alpha_1,\dots,\alpha_r)\in\C^r$ and for $i\in\{1,\dots,r\}$, we write 
\[\hat\bal_i=(\alpha_1,\dots,\alpha_{i-1},0,\alpha_{i+1},\dots,\alpha_r).\]
We also write by
\[i_{\bal}\colon \{\bal\}\hookrightarrow \C^r.\]
the closed embedding and for $r\ge 2$ set %$V_{\alpha}\subseteq \C$ a sufficiently small neighborhood of $\alpha$ and by 
\[\delta_{\hat\bal_i}\colon \C\hookrightarrow \C^r,\quad x\mapsto (\alpha_1,\dots,\alpha_{i-1},x,\alpha_{i+1},\dots,\alpha_r).\]
For $\bal=(\alpha_1,\dots,\alpha_r)\in\C^r$, the multivalued holomorphic function $F^\bal$ determines a rank one local system $L_\blamb$ on $X\setminus D$, that is, the local system has local monodromy along each $D_i$ given by multiplication by $\lambda_i$ with 
\[\lambda_i=exp(-2\pi \sqrt{-1}\alpha_i).\]
By Riemann-Hilbert correspondence, $Rj_*L_\bal$ determines a regular holonomic $\shD_X$-module, denoted by $\cV_\bal$, where $j\colon X\setminus D\hookrightarrow X$ the open embedding.
We write 
\[\cM_\bal=\cM\otimes_\sO \cV_\bal.\]

\begin{prop}\label{prop:!*gnbpbtopoint}
Let $\cM$ be a holonomic $\shD_X$-module. With notations as above, for each $\bal\in\C^r$ and for each $i$, we have 
\[\bL\wt i_{\bal}^*(\cM(*D_F^{(r)}))\stackrel{q.i.}{\simeq}\cM_\bal(*D) \quad\textup{and}\quad  \bL \wt i_{\bal}^*(\cM(!D_F^{(r)}))\stackrel{q.i.}{\simeq}\cM_\bal(!D)\]
and for each $0<i\le r$
\[\bL\wt \delta_{\hat\bal_i}^*(\cM(*D_F^{(r)}))\stackrel{q.i.}{\simeq}\cM_{\hat\bal_i}(*D_{f_i}^{(1)})\quad\textup{and}\quad\bL\wt \delta_{\hat\bal_i}^*(\cM(!D_F^{(r)}))\stackrel{q.i.}{\simeq} \delta_{\hat\bal_i}^*(\cM(!D_F^{(r)})).\]
%and 
%\[\textup{Im}\big(\delta_{\hat\bal_i}^*(\cM(!D_F^{(r)}))\to \delta_{\hat\bal_i}^*(\cM(*D_F^{(r)}))\big)=\cM_{\hat\bal_i}(!D_{f_i}^{(1)})\]
\end{prop}
\begin{proof}
By construction, $\sM$ is flat over $R\simeq\C[\bs]$. Then 
\[\bL\wt i_{\bal}^*(\cM(*D_F^{(r)}))\stackrel{q.i.}{\simeq}\wt i_{\bal}^*(\cM(*D_F^{(r)}))=\sM\otimes_R \C_\bal=\cM_\bal(*D),\]
where $\C_\bal$ is the residue field of $\bal\in\C^r$.
By using Lemma \ref{lm:basechangeanddual} and Cohen-Macaulayness, 
\[\bL\wt i_{\bal}^*(\cM(!D_F^{(r)}))\stackrel{q.i.}{\simeq}\wt i_{\bal}^*(\cM(!D_F^{(r)}))=\cM_\bal(!D)\]
follows similarly. 

Similarly, we get 
\[\bL\wt \delta_{\hat\bal_i}^*(\cM(*D_F^{(r)}))\stackrel{q.i.}{\simeq}\cM_{\hat\bal_i}(*D_{f_i}^{(1)})\quad\textup{and}\quad\bL\wt \delta_{\hat\bal_i}^*(\cM(!D_F^{(r)}))\stackrel{q.i.}{\simeq} \delta_{\hat\bal_i}^*(\cM(!D_F^{(r)})).\]
\end{proof}
Let us remark that by definition, $\delta_{\hat\bal_i}^*(\cM(!D_F^{(r)}))$ is the minimal extension of 
$$\delta_{\hat\bal_i}^*(\cM(*D_F^{(r)}))|_{(X\setminus D)\times \C}=\cM_{\hat\bal_i}(*D_{f_i}^{(1)})|_{(X\setminus D)\times \C}$$ along $D\times \C$ by construction, which in general is not necessarily $\cM_{\hat\bal_i}(!D_{f_i}^{(1)})$, the minimal extension of $$\cM_{\hat\bal_i}(*D_{f_i}^{(1)})|_{(X\setminus (f_i=0))\times \C}$$
along $(f_i=0)\times\C$. But we still have the following corollary:
%The other required statements for $i_{\hat\bal_i}$ can be checked in a similar way. We leave details for interested readers. 

%The proposition above has an immediate corollary:
\begin{coro}\label{coro:basechageSF}
Let $\cM$ be a holonomic $\shD_X$-module. Then for each $i$ and for very general $\bal\in (g\cdot (L\cdot \bs+\alpha)=0)$, an irreducible component of $\supp_{\C^r}(\Psi_F(\cM))$ , if $L\not=e_i$, then we have
\[\bL\wt \delta_{\hat\bal_i}^*\Psi_F(\cM)\stackrel{q.i.}{\simeq} \Psi_{f_i}(\cM_{\hat\bal_i}),\]
where $e_i$ is the $i$-th unit vector in $\Z_{\ge0}^r$.
\end{coro}
\begin{proof}
    Since $\bal\in (g\cdot (L\cdot \bs+\alpha)=0)$ is very general, the condition $L\not= e_i$ gives 
    \[\delta_{\hat\bal_i}^*(\cM(*D_F^{(r)}))|_{(X\setminus (f_i=0))\times \C}=\delta_{\hat\bal_i}^*(\cM(!D_F^{(r)}))|_{(X\setminus (f_i=0))\times \C}=\cM_{\hat\bal_i}(*D_{f_i}^{(1)})|_{(X\setminus (f_i=0))\times \C}.\]
    Thus, we have 
    \[\cM_{\hat\bal_i}(!D_{f_i}^{(1)})=\delta_{\hat\bal_i}^*(\cM(!D_F^{(r)})).\]
    By using Proposition \ref{prop:!*gnbpbtopoint}, the proof is done.
\end{proof}

\begin{proof}[Proof of Theorem \ref{thm:relccgny}]
The relative characteristic cycle formula for $\cM(*D_F^{(r)})$ has been obtained in Theorem \ref{thm:relholmaxexcc} even if $\cM$ is only holonomic. By \eqref{eq:locj!st}, we also get the relative characteristic cycle of $\cM(!D_F^{(r)})$.

We now prove the relative characteristic cycle formula for $\Psi_F(\cM)$. For $L\in S(F,\cM)$ and $\alpha\in\kappa(L)$, we fix a relatively compact open subset $V\subseteq \C^r$ intersecting $(L\cdot \bs+\alpha=0)$. By Theorem \ref{thm:relsd} and Proposition \ref{lm:altrelhol}, we can assume that 
\[\Lambda_{L,\alpha}\times (L\cdot \bs+\alpha=0)|_V\]
is the part of $\CC^\rel(\Psi_F(\cM))|_{W\times V}$ over $(L\cdot \bs+\alpha=0)$, where $\Lambda_{L,\alpha}$ is a conic Lagrangian cycle supported on $\Ch(\cM(*D))$ over $D\cap W$. Using the ``gluing" date of $\cM(*D_F^{(r)})$ and $\cM(!D_F^{(r)})$ in \S \ref{subsec:gluemmex}, we can move $\Lambda_{L,\alpha}\times (L\cdot \bs+\alpha=0)|_V$ along $(L\cdot \bs+\alpha=0)$ as we translate $V$ along $(L\cdot \bs+\alpha=0)$ such that $\Lambda_{L,\alpha}\times (L\cdot \bs+\alpha=0)|_V$ is always the part of $\CC^\rel(\Psi_F(\cM))|_{W\times V}$ over $(L\cdot \bs+\alpha=0)$. Therefore, 
\[\Lambda_{L,\alpha}\times (L\cdot \bs+\alpha=0)\]
is the part of $\CC^\rel(\Psi_F(\cM))|_{W\times \C^r}$ over $(L\cdot \bs+\alpha=0)$. 

For every $g\in G$, using the isomorphism \[G\simeq\Z^r\]
we can assume $g=g_\ba$ for some $\ba\in\Z^r$. Since by \eqref{eq:alptriso} for every $k>0$ we have 
\[\tau^{-1}_\ba(\wt{\sN_{k}/\sN_{-k}})\simeq \wt{\sN_{k}/\sN_{-k}},\]
the part of $\CC^\rel(\Psi_F(\cM))|_{W\times \C^r}$ over $g\cdot(L\cdot \bs+\alpha=0)$ is 
\[\Lambda_{L,\alpha}\times (L\cdot \bs+\alpha=0).\]
Therefore,
\[\CC^\rel(\Psi_F(\cM))=\sum_{g\in G}\sum_{L\in S(F,\cM)}\sum_{\alpha\in \tilde\kappa(L)} \Lambda_{L,\alpha}\times (g\cdot(L\cdot s+\alpha)=0).\]
To avoid counting components repeatedly, the summation above is taken over $\wt\kappa(L)$ (instead of $\kappa(L)$).
\end{proof}

We now give a formula for $\Lambda_{L,\alpha}$.
\begin{prop}\label{prop:lambdaLalf1dim}
In the situation of Theorem \ref{thm:relccgny}, for each $L\in S(F,\cM)$ and for each $\alpha\in\wt\kappa(L)$, then after picking a very general point 
$$\bal=(\alpha_1,\dots,\alpha_r)\in(L\cdot\bs+\alpha=0)$$ for each $i\in\{1,\dots, r\}$ such that $L\not= e_i$ we have
\[\Lambda_{L,\alpha}=\CC(\Psi_{f_i}(\cM_{\hat\bal_i})|_{W\times V_{\alpha_i}})\]
where $V_{\alpha_i}$ is a sufficiently small neighborhood of $\alpha_i\in\C$.
\end{prop}
\begin{proof}
Let us first explain the notation 
\[\CC(\Psi_{f_i}(\cM_{\hat\bal_i})|_{W\times V_{\alpha_i}}).\]
In the case $r=1$, $\Psi_{f_i}(\cM_{\hat\bal_i})$ is a relative holonomic $\shD$-module over the complex line $\C$. In this case, $S(f_i,\cM_{\hat\bal_i})=\{1\}$. By Theorem \ref{thm:relsd}, we then know $\Psi_{f_i}(\cM_{\hat\bal_i}$) is supported over an infinite discrete subset of $\C$, where the discrete subset containing $\alpha_i$ (by Corollary \ref{coro:basechageSF}).  With $V_{\alpha_i}$ small enough, $\Psi_{f_i}(\cM_{\hat\bal_i})|_{W\times V_{\alpha_i}}$ is just supported over $\alpha_i$. More explicitly, by construction $\Psi_{f_i}(\cM_{\hat\bal_i})|_{W\times V_{\alpha_i}}$
is just
\[\frac{(\sN^{\hat\bal_i}_k)_{m_{\alpha_i}}}{(\sN^{\hat\bal_i}_{-k})_{m_{\alpha_i}}}\]
for $k\gg 0$ (since both of them are killed by $(s+\alpha_i)^l$ for some $l>0$, sheafification does nothing),
where $\sN^{\hat\bal_i}_k$ is the $\shD_{X,\C[s_i]}$-module $\sN_k$ for $\cM_{\hat\bal_i}$ as in \S \ref{subsec:3notations}, and $m_{\alpha_i}$ is the maximal ideal of $\alpha_i\in\C$. But the latter module is particularly a $\shD_X$-module. Thus, by Lemma \ref{lm:altrelhol} 
\be\label{eq:ccnb1dim}
\CC^{\rel}(\Psi_{f_i}(\cM_{\hat\bal_i})|_{W\times V_{\alpha_i}})=\CC(\Psi_{f_i}(\cM_{\hat\bal_i})|_{W\times V_{\alpha_i}})\times \{\alpha_i\}.
\ee
As a consequence, $\Psi_{f_i}(\cM_{\hat\bal_i})|_{W\times V_{\alpha_i}}$ is a holonomic $\shD_X$-module (in fact, it is also regular if so is $\cM$). Moreover, if we use the terminology in \cite{KasV,MalV,Wubf}, $\Psi_{f_i}(\cM_{\hat\bal_i})|_{W\times V_{\alpha_i}}$ is the $\alpha_i$-nearby cycle of $\cM_{\hat\bal_i}$.

We now prove the required identity. By the proof of Theorem \ref{thm:relccgny}, to obtain $\Lambda_{L,\alpha}$ it suffices to consider a small neighborhood $V$ intersecting $(L\cdot\bs+\alpha=0)$. By the definition of characteristic cycles, we only need to consider a small neighborhood of a general point $\bal\in V\cap (L\cdot\bs+\alpha=0)$, which is equivalent to globally considering a small neighborhood $V_{\bal}$ of a very general point $\bal\in (L\cdot\bs+\alpha=0)$. By Corollary \ref{coro:basechageSF}, if $L\not=e_i$, then we have
\[\wt \delta_{\hat\bal_i}^*(\Psi_F(\cM))=\Psi_{f_i}(\cM_{\hat\bal_i}).\]
By Cohen-Macaulayness and Proposition \ref{prop:suppsubcmrh}, we can inductively apply \cite[Proposition 2.7]{Wuch} (by picking $r-1$ general hyperplanes passing $\bal$) and conclude 
\be\label{eq:ccnbcmm}
\wt \delta_{\hat\bal_i}^*(\CC^\rel(\Psi_F(\cM)))=\CC^\rel(\wt \delta_{\hat\bal_i}^*(\Psi_F(\cM)))=\CC^\rel(\Psi_{f_i}(\cM_{\hat\bal_i})).\ee
By the very general choice of $\bal$, we have 
\[\CC^\rel(\Psi_F(\cM))|_{W\times V_{\bal}}=\Lambda_{L,\alpha}\times (L\cdot \bs+\alpha=0),\]
Taking $V_{\alpha_i}=\delta_{\hat\bal_i}^{-1}(V_{\bal})$, we finish the proof by \eqref{eq:ccnb1dim} and \eqref{eq:ccnbcmm}.
\end{proof}

\subsection{A linearity conjecture}\label{subsec:linearextsuppconj}
Let $\cM$ be a regular holonomic $\shD_X$-module. By Proposition \ref{prop:j_*ncm}, the maximal extension $\cM(*D_F^{(r)})$ is $n$-Cohen-Macaulay. On the contrary, for some $k$ (and hence for all $k$ by applying $\tau$-translation) $\wt\sN_k$ are not $n$-Cohen-Macaulay in general. Thus, we might have non-zero 
\[\Ext^l_{\wt\shD_{X,R}}(\wt\sN_k,\wt\shD_{X,R}) \textup{ for } n\le l\le n+r\]
and hence non-zero right $\shD_{X,R}$-modules
\[\Ext^l_{\shD_{X,R}}(\sN_k,\shD_{X,R}) \textup{ for } n\le l\le n+r.\]

Motivated by Corollary \ref{cor:regholzbfnontr} and the relative condimension filtration in \S \ref{subsec:relcodimfil}, we conjecture:
\begin{conj}\label{conj:linarityextsupp}
Let $\cM$ be a regular holonomic $\shD_X$-module. If locally on a relatively compact open subset $W\subseteq X$
\[\Ext^l_{\shD_{X,R}}(\sN_k,\shD_{X,R})\not=0 \textup{ for some } n\le l\le n+r,\]
then we have:
\begin{enumerate}[label=(\roman*)]
    \item If $S_j(\Ext^l_{\shD_{X,R}}(\sN_k,\shD_{X,R}))\not=\emptyset$, then it is a finite union of translated $(j-n)$-codimensional linear subspaces of $\C^r$.
    %\item Modulo $G$-action, $\supp_{\C^r}\Ext^l_{\shD_{X,R}}(\sN_k,\shD_{X,R})$ is independent of choices of $\cM_0$ and $k$.
    \item If $\cM$ underlies a $\Q$-mixed Hodge module $($for example $\cM=\sO_X)$, then all $S_j(\Ext^l_{\shD_{X,R}}(\sN_k,\shD_{X,R}))$ are defined over $\Q$.
\end{enumerate}
\end{conj}
By Auslander regularity, Lemma \ref{lm:angrnchreq} and Theorem \ref{thm:relholmaxexcc}, we have 
\[\Ch^\rel(\Ext^n_{\wt\shD_{X,R}}(\wt\sN_k,\wt\shD_{X,R}))=\Ch^\rel(\wt\sN_k)=\Ch(\cM(*D))\times\C^r\]
and that $\Ext^n_{\shD_{X,R}}(\sN_k,\shD_{X,R})$ is pure of codimension $n$. Thus, the conjecture above holds for $l=n$.

%Because of the short exact sequence
%\[0\to \sN_{-1}\rightarrow \sN_0\to \frac{\sN_0}{\sN_{-1}}\to0,\]
%we also conjecutre: 
%\begin{conj}\label{conj:linearzbfgeneral}
%Let $\cM$ be a regular holonomic $\shD_X$-module. Locally on a relatively compact open subset $W\subseteq X$
%we have:
%\begin{enumerate}[label=(\roman*)]
   % \item If $S_j(\sN_0/\sN_{-1})\not=\emptyset$, then it is a finite union of translated $(j-n)$-codimensional linear subspaces of $\C^r$.
    %\item Modulo $G$-action, $\supp_{\C^r}\Ext^l_{\shD_{X,R}}(\sN_k,\shD_{X,R})$ is independent of choices of $\cM_0$ and $k$.
   % \item If $\cM$ underlies a $\Q$-mixed Hodge module $($for example $\cM=\sO_X)$, then $Z(B_F(\cN_0))$ is defined over $\Q$.
%\end{enumerate}
%\end{conj}

%When $\cM_0=\sO_X$, by Proposition \ref{prop:zbfsupp=}, Conjecture \ref{conj:linearzbfgeneral} is a weak version of \cite[Conjecture 1.1]{Budur}.
%Since 
%\[Z_{r-1}(B_F(\sN_0))=S_{n+1}(\sN_0/\sN_{-1}),\]
%Conjecture \ref{conj:linearzbfgeneral} (i) holds for $j=n+1$. Conjecture \ref{conj:linearzbfgeneral} (ii) also holds in this case because the quasi-unipotentcy of $\Q$-Hodge modules implies $\kappa(L)\subseteq \Q$ for $L\in S(F,\cM)$ (see \cite{Sab2}). 

When $\cM_0=\sO_X$, Budur made a conjecture about the structure of generators of the Bernstein-Sato ideal $B_F$ \cite[Conjecture 1.1]{Budur}. The following is a weak version of the conjecture.

\begin{conj}[Budur]\label{conj:linearzbfgeneral}
Each irreducible component of $Z(B_F)$ is a translated linear subspace of $\C^r$ defined over $\Q$.
\end{conj}
By Sabbah \cite{Sab2} (or one can apply Corollary \ref{cor:regholzbfnontr}) and \cite{Gyo}, codimension-one irreducible components of $Z(B_F)$ are translated linear subspaces of $\C^r$ defined over $\Q$.

\begin{prop}\label{prop:myconjimplybudur}
Conjecture \ref{conj:linarityextsupp} $\Longrightarrow$ Conjecture \ref{conj:linearzbfgeneral}.
\end{prop}
\begin{proof}
We assume $E$ an irrducible component of $Z(B_F)$ of codimension $2\le l\le r$. We fix a point $(x,p)$ such that $(\wt\sN_0/\wt\sN_{-1})_{(x,p)}\not=0$ (by Proposition \ref{prop:zbfsupp=} and $\wt\sN_0/\wt\sN_{-1}= \wt{\sN_0/\sN_{-1}}$) and $p$ is a general point in $E$. By Corollary \ref{cor:algdecompofchrel} and the dimension of $E$, we can choose a small neighborhood $W$ of $x\in X$ and a Zariski open neighborhood $V\subseteq\C^r$ of $p$
such that the codimension filtration satisfies 
\[T_{j}(\wt\sN_0/\wt\sN_{-1})=\wt\sN_0/\wt\sN_{-1} \]
for $j<n+l$ over $W\times V$.
By Lemma \ref{lm:angrnchreq} and Corollary \ref{cor:puritylocgcomp}, the Auslander regularity tells
\[\textup{codim}_\C\Ch^\rel(\Ext^{n+j}_{\wt\shD_{X,R}}(\wt\sN_0/\wt\sN_{-1},\wt\shD_{X,R})|_{W\times V})> l \textup{ for $j>l$}.\]
By shrinking $V$ in the Zariski topology (pick a more general $p\in E\cap V$ if necessary), we can assume that on $W\times V$
$$T_{n+l}(\wt\sN_0/\wt\sN_{-1})=\wt\sN_0/\wt\sN_{-1}$$
and $\wt\sN_0/\wt\sN_{-1}$ is $(n+l)$-Cohen-Macaulay.

Now, we consider the short exact sequence (on $W$)
\[0\to \sN_{-1}\rightarrow \sN_0\to \frac{\sN_0}{\sN_{-1}}\to0.\]
Taking analytic sheafification and dual, we have a long exact sequence \[\dots\to\Ext^{n+l-1}_{\wt\shD_{X,R}}(\wt\sN_{-1},\wt\shD_{X,R})\to  \Ext^{n+l}_{\wt\shD_{X,R}}(\frac{\wt\sN_0}{\wt\sN_{-1}},\wt\shD_{X,R})\rightarrow \Ext^{n+l}_{\wt\shD_{X,R}}(\wt\sN_0,\wt\shD_{X,R})\rightarrow\cdots.\]
Since $\wt\sN_0$ and $\wt\sN_{-1}$ are $n$-pure (since they are submodules of a $n$-Cohen-Macaulay module $\wt\sM$ by Proposition \ref{prop:j_*ncm} and Proposition \ref{prop:suppsubcmrh}), by \cite[A.IV.2.6]{Bj}, Corollary \ref{cor:puritylocgcomp} and Lemma \ref{lm:angrnchreq} we have 
\[T_{n+l+1}(\Ext^{n+l}_{\wt\shD_{X,R}}(\wt\sN_0,\wt\shD_{X,R}))=\Ext^{n+l}_{\wt\shD_{X,R}}(\wt\sN_0,\wt\shD_{X,R})\]
and
\[T_{n+l}(\Ext^{n+l-1}_{\wt\shD_{X,R}}(\wt\sN_0,\wt\shD_{X,R}))=\Ext^{n+l-1}_{\wt\shD_{X,R}}(\wt\sN_0,\wt\shD_{X,R}).\]
Since $\wt\sN_0/\wt\sN_{-1}$ is $(n+l)$-Cohen-Macaulay on $W\times V$, we know 
\[\Ch^\rel(\Ext^{n+l-1}_{\wt\shD_{X,R}}(\frac{\wt\sN_0}{\wt\sN_{-1}},\wt\shD_{X,R}))=\Ch^\rel(\frac{\wt\sN_0}{\wt\sN_{-1}})\]
on $W\times V$.
Therefore, 
\[\Ch^\rel(\Ext^{n+l-1}_{\wt\shD_{X,R}}(\frac{\wt\sN_0}{\wt\sN_{-1}},\wt\shD_{X,R}))\subseteq \Ch^\rel(\frac{T_{n+l}(\Ext^{n+l-1}_{\wt\shD_{X,R}}(\wt\sN_0,\wt\shD_{X,R}))}{T_{n+l+1}(\Ext^{n+l-1}_{\wt\shD_{X,R}}(\wt\sN_0,\wt\shD_{X,R}))}).\]
 By Proposition \ref{prop:zbfsupp=} and Proposition-Definition \ref{prop-def:relsupp}, we have \[E\subseteq S_{n+l}(\Ext^{n+l-1}_{\wt\shD_{X,R}}(\wt\sN_0,\wt\shD_{X,R})),\]
 which proves the required implication. 
\end{proof}

\section{Alexander complex of Sabbah}\label{sec:recalalexsab}
In this section, we recall the construction of the Alexander complexes in \cite{Sab90}.
\subsection{Alexander complex of Sabbah}\label{subsec:AcomSab}
We keep using notations in \S \ref{subsec:3notations}. The universal Alexander sheaf is defined by 
\[\psi^{\textup{univ}}=Exp_!(\C_{\C^r})\]
where $Exp_!$ is the proper direct image functor of $Exp$ (cf. \cite[2.5]{KSbook}) and $\C_{\C^r}$ is the $\C$-constant sheaf on $\C^r$. Since $Exp$ gives the universal covering of $(\C^*)^r$, $\psi^{\textup{univ}}$ is a locally free sheaf of $\C[G]$-modules. We use the isomorphism 
\[\C[G]\simeq\C[t_1^{\pm},\dots,t_r^{\pm}]\]
with $t_i$ representing the (counterclockwise) loops around the puncture of each factor of $(\C^*)^r$. Then the monodromy of $\psi^{\textup{univ}}$ around  each loop is induced by the $t_i$-multiplication on $\C[t_1^{\pm},\dots,t_r^{\pm}]$. 
For $\cF^\bullet\in D^b_c(\C_X)$, its \emph{Alexander complex} along $F$ (also known as \emph{Sabbah specialization complex}) is defined as 
\[\psi_F(\cF^\bullet)=i_{D,*}i_D^{-1}Rj_*(\cF^\bullet|_{X\setminus D}\otimes_\C F^{-1}\psi^{\textup{univ}}),\]
where 
%$F\colon X\setminus D\to ((\C^*)^r)$ is the morphism given by $F=(f_1,\dots,f_r)$, 
$j\colon X\setminus D\hookrightarrow X$ is the open embedding and $i_D\colon D\hookrightarrow X$ is the closed embedding. More generally, for every subset $I\subseteq \{1,2,\dots,r\}$, the \emph{Sabbah specialization complex} along $D_I=\bigcap_{i\in I}D_i$ is define as
\[\psi_{D_I}(\cF^\bullet)=i_{D_I,*}i_{D_I}^{-1}Rj_*(\cF^\bullet|_{X\setminus D}\otimes_\C F^{-1}\psi^{\textup{univ}}),\]
where $i_{D_I}\colon D_I\hookrightarrow X$ is the closed embedding.
Since $\psi^{\textup{univ}}$ is locally free over $\C[G]$, $\psi_{D_I}(\cF^\bullet)\in D^b_c(\C[G])$. 

%\subsection{The universal relative local system}\label{subsec:univrellocs}
%We define a locally free $pr^{-1}\sO_{(\C^*)^r}$-module on the complex manifold $(\C^*)^r\times (\C^*)^r$ with $pr\colon (\C^*)^r\times (\C^*)^r\to (\C^*)^r$ the second projection, denoted by 

\section{Comparison}\label{sec:comparison}
\subsection{The universal relative flat bundle}\label{subsec:univrellocs}
On $\C^r_{\mathbf{x}}\times\C^r$ with $\mathbf{x}=(x_1,\dots,x_r)$ the coordinates, we write 
\[D_{\mathbf x}=(\prod_{i=1}^r x_i=0) \quad \textup{and}\quad U=\C^r\setminus D.\]
Then using the recipe from \S \ref{sec:*!ext}, we have the maximal extension $\sO_{\C^r_{\mathbf{x}}}(*D_{\mathbf x}^{(r)})$. We set 
\[\sO^{\textup{univ}}=\sO_{\C^r_{\mathbf{x}}}(*D_{\mathbf x}^{(r)})|_{U\times\C^r },\]
calling it the universal relative flat bundle on $U\times\C^r$ because 
\be\label{eq:unvimaxexpb}
\wt F^*(\sO_{\C^r_{\mathbf{x}}}(*D_{\mathbf x}^{(r)}))=\sO_X(*D_F^{(r)})
\ee
in the situation of \S \ref{subsec:3notations}, with $\wt F=(F,\textup{id})\colon X\times\C^r\to \C^r\times \C^r$ the associated morphism, or equivalently 
\[F^*(\sO_{\C^r_{\mathbf x}}(*D_{\mathbf x})\otimes_\C \C[\bs]\cdot {\mathbf x}^\bs)= \sO_X(*D_F)\otimes_\C \C[\bs]\cdot F^\bs.\]
\begin{remark}
Since $\mathbf x$ is the complex coordinates of $\C^r_{\mathbf x}$, it is a direct computation to get the relative characteristic cycle:
\[\CC^{\rel}\big(\sO_{\C^r_{\mathbf{x}}}(*D_{\mathbf x}^{(r)})\big)=(\sum_{\beta} T^*_{\bar X_\beta}\C^r_{\mathbf x})\times \C^r,\]
where $\bigsqcup_{\beta}X_\beta=\C^r_{\mathbf x}$ is the stratification naturally given by the coordinates $\mathbf x$. In particular, $\sO_{\C^r_{\mathbf{x}}}(*D_{\mathbf x}^{(r)})$ is relative holonomic. In fact, it is regular relative holonomic (see \cite[\S 2.1]{FSO} for the definition of relative regularity). Therefore, if $\cM$ is a regular holonomic $\shD_X$-module, then the relative holonomicity of $\cM(*D_F^{(r)})$ in Theorem \ref{thm:relholmaxexcc} can also be deduced from \cite[Theorem 2]{FFS19}.
\end{remark}

To avoid confusion, we set the morphism 
\[\pi_{\mathbf x}=(\textup{id},Exp)\colon U\times\C^r\rightarrow U\times (\C^*)^r.\]
\begin{lemma}\label{lm:relbabyRH}
We have
    \[\DR_{U\times\C^r/\C^r}(\sO^{\textup{univ}})\stackrel{q.i.}{\simeq}\cH^{-n}\DR_{U\times\C^r/\C^r}(\sO^{\textup{univ}})[n]\]
\end{lemma}
\begin{proof}
    We can decompose $U\times\C^r/\C^r=(\C^*_{x_1}\times\C)\times\cdots (\C^*_{x_r}\times\C)$. It is thus sufficient to prove 
    \[\DR_{\C^*_x\times\C/\C}(\sO^{\textup{univ}}_{\C_x\times \C})\stackrel{q.i.}{\simeq}\cH^{-1}\DR_{\C^*_x\times\C/\C}(\sO^{\textup{univ}}_{\C_x\times \C})[1].\]
    In this case, since we are working on $\C^*_x\times\C$
    \[\DR_{\C^*_x\times\C/\C}(\sO^{\textup{univ}}_{\C_x\times \C})=[\sO^{\textup{univ}}_{\C_x\times \C}\xrightarrow{x\partial_{x}}\sO^{\textup{univ}}_{\C_x\times \C}].\]
    By direct computation, we know that $x\partial_x$ acts on $\sO^{\textup{univ}}_{\C_x\times \C}$ surjectively. Moreover, its kernel is 
    \[x^{-s}\cdot p^{-1}\cO_\C\hookrightarrow \sO^{\textup{univ}}_{\C_x\times \C}\]
    such that $x^{-s}\mapsto x^{-s}\cdot x^s$ (notice that $x^{-s}$ is a global section of $\cO_{\C^*_x\times \C}$ and that $x^s$ is the generator of $\sO^{\textup{univ}}_{\C_x\times \C}$), where $p: \C^*_x\times\C\to \C$ is the second projection.
\end{proof}
By Lemma \ref{lm:DReq}, $\DR_{U\times\C^r/\C^r}(\sO^{\textup{univ}})$ and 
\[\cH^{-n}\DR_{U\times\C^r/\C^r}(\sO^{\textup{univ}})\simeq{\mathbf x}^{-\bs}\cdot p^{-1}\cO_{\C^r}\]
are $G$-equivariant. Then we define 
\[\cL^{\textup{univ}}\coloneqq \pi_{{\mathbf x},*}^G({\mathbf x}^{-\bs}\cdot p^{-1}\cO_{\C^r}).\]
By construction, $\cL^{\textup{univ}}$ is characterized by the following two conditions:
\begin{enumerate}[label=(\roman*)]
    \item  $\cL^{\textup{univ}}|_{W\times (\C^*)^r}\simeq pr^{-1}\sO_{(\C^*)^r}|_{W\times (\C^*)^r}$ for all simply connected open subsets $W\subset (\C^*)^r$
    \item for every $\blamb=(\lambda_1,\dots,\lambda_r)\in (\C^*)^r$, $\wt i_{\blamb}^*\cL^{\textup{univ}}$ is the local system on the first factor of $(\C^*)^r\times (\C^*)^r$ satisfying that the monodromy along each $t_i$ is given by multiplication by $\lambda_i$, where $i_{\blamb}\colon \{\blamb\}\hookrightarrow (\C^*)^r$ is the closed embedding into the second factor of $(\C^*)^r\times (\C^*)^r$.
\end{enumerate}
\begin{lemma}\label{lm:univreldm}
We have
\[\textup{$\pi_{{\mathbf x},*}^G(\DR_{U\times\C^r/\C^r}(\sO^{\textup{univ}}))\stackrel{q.i.}{\simeq} \cL^{\textup{univ}}[n]$,  $\sO_{U\times\C^r}\otimes_{p^{-1}\sO_{\C^r}}\pi^{-1}_{\mathbf x}(\cL^{\textup{univ}})\simeq \sO^{\textup{univ}}$}.\]
and $\cL^{\textup{univ}}\simeq \wt\psi^{\textup{univ}}.$
\end{lemma}
\begin{proof}
   The first quasi-isomorphism follows by definition. The second isomorphism is given by \cite[Part I.0.3.Lemma]{BL} and relative Riemann-Hilbert Correspondence. Lastly, similar to the proof of Lemma \ref{lm:relbabyRH}, $\cL^{\textup{univ}}\simeq \wt\psi^{\textup{univ}}$ can be reduced to the case $r=1$ and then checked directly. 
\end{proof}
\subsection{Proof of Theorem \ref{thm:RHALXhigh}}
We first prove the following theorem.
\begin{theorem}\label{thm:RHj_*}
Let $\cM$ be a regular holonomic $\shD_X$-module. Then naturally
$$\pi_{*}^G(\DR_{X\times\C^r/\C^r}(\cM(*D_F^{(r)})))\stackrel{q.i.}{\simeq} \wt{Rj_*}(\DR(\cM)|_{X\setminus D}\otimes_\C F^{-1}\psi^{\textup{univ}}).$$
\end{theorem}
\begin{proof}
By using the second isomorphism in Lemma \ref{lm:univreldm} and \eqref{eq:unvimaxexpb}, we get 
\be\label{eq:openkeyisouniv}
\DR_{X\times\C^r/\C^r}(\cM(*D_F^{(r)}))|_{(X\setminus D)\times \C^r}\simeq p_1^{-1}(\DR(\cM|_{X\setminus D}))\otimes_\C \wt F^{-1}(\pi^{-1}_{\mathbf x}(\cL^{\textup{univ}})),
\ee
where $p_1\colon X\times \C^r\to X$ is the first projection. Since $G$ acts trivially on objects originating from $X$ (for instance $\cM$ and $p_1^{-1}(\DR(\cM|_{X\setminus D}))$), by \cite[Part I.0.3.Lemma]{BL}, \eqref{eq:openkeyisouniv} is a $G$-equivariant isomorphism. Therefore, by the third isomorphism in Lemma \ref{lm:univreldm}, we have
\[\pi_{*}^G(\DR_{X\times\C^r/\C^r}(\cM(*D_F^{(r)})))|_{(X\setminus D)\times (\C^*)^r}\simeq \wt{\DR(\cM)|_{X\setminus D}\otimes_\C F^{-1}\psi^{\textup{univ}}}.\]
By adjunction, the isomorphism above induces a morphism 
\be\label{eq:keynmrdr}
\pi_{*}^G(\DR_{X\times\C^r/\C^r}(\cM(*D_F^{(r)})))\rightarrow \wt{Rj_*}(\DR(\cM)|_{X\setminus D}\otimes_\C F^{-1}\psi^{\textup{univ}}).
\ee
We then make derived pullback of the morphism above by $\wt i_{\blamb}$ for every $\blamb\in(\C^*)^r$. But by definition 
\[\bL\wt i_{\blamb}^*\pi_{*}^G(\DR_{X\times\C^r/\C^r}(\cM(*D_F^{(r)})))\simeq \bL\wt i_{\bal}^*\DR_{X\times\C^r/\C^r}(\cM(*D_F^{(r)}))\]
for any $\bal\in Exp^{-1}(\blamb)$. Since $\cM$ is regular, by Proposition \ref{prop:!*gnbpbtopoint} 
\[\bL\wt i_{\bal}^*(\DR_{X\times\C^r/\C^r}(\cM(*D_F^{(r)})))\simeq \DR(\bL\wt i_{\bal}^*(\cM(*D_F^{(r)}))\simeq Rj_*(\DR(\cM)|_{X\setminus D}\otimes L_{\blamb}).\]
Moreover, by construction
\[\bL\wt i_{\blamb}^*\wt{Rj_*}(\DR(\cM)|_{X\setminus D}\otimes_\C F^{-1}\psi^{\textup{univ}})\simeq {Rj_*}(\DR(\cM)|_{X\setminus D}\otimes_\C F^{-1}\psi^{\textup{univ}})\otimes^\bL_{\C[G]} \C_{\blamb},\]
where $\C_{\blamb}$ is the residue field of $\blamb\in(\C^*)^r$. Since $\psi^{\textup{univ}}$ is locally free over $\C[G]$ and hence over $\C$,
\[{Rj_*}(\DR(\cM)|_{X\setminus D}\otimes_\C F^{-1}\psi^{\textup{univ}})\otimes_{\C[G]} \C_{\blamb}\simeq {Rj_*}(\DR(\cM)|_{X\setminus D}\otimes_\C F^{-1}\psi^{\textup{univ}}\otimes_{\C[G]} \C_{\blamb}).\]
By the definition of $\psi_{\textup{univ}}$,
\[F^{-1}\psi^{\textup{univ}}\otimes_{\C[G]} \C_{\blamb}\simeq L_\blamb.\]
Therefore, the $\bL\wt i_{\blamb}^*$ pullback of the morphism \eqref{eq:keynmrdr} is an isomorphism for each $\blamb$. By \cite[Proposition 2.2 and Theorem 3.7]{FSO}, the morphism \eqref{eq:keynmrdr} is isomorphic. \end{proof}

Although Theorem \ref{thm:RHALX} is a special case of Theorem \ref{thm:RHALXhigh} (by Corollary \ref{cor:dualalcop}), we first give a direct proof of Theorem \ref{thm:RHALX} as a warm-up. We keep using notation in the proof of Theorem \ref{thm:RHj_*}.
By Lemma \ref{lm:!*incl}, 
\[\Psi_F(\cM)\stackrel{q.i.}{\simeq}\textup{cone}(\cM(!D_F^{(r)})\rightarrow \cM(*D_F^{(r)})).\]
By Proposition \ref{prop:!*gnbpbtopoint}, 
\[\bL\wt i_{\bal}^*\textup{cone}(\cM(!D_F^{(r)})\rightarrow \cM(!D_F^{(r)}))\stackrel{q.i.}{\simeq} \textup{cone}(\cM_\bal(!D)\to\cM_\bal(*D)).\]
Taking the relative de Rham functor, we then have that
\[
\DR_{X\times\C^r/\C^r}\bL\wt i_{\bal}^*\textup{cone}(\cM_\bal(!D_F^{(r)})\rightarrow \cM_\bal(!D_F^{(r)}))
\]
is quasi-isomorphic to 
\[ \sB^\bullet\coloneqq\textup{cone}(j_!(\DR(\cM)|_{X\setminus D}\otimes L_{\blamb})\to Rj_*(\DR(\cM)|_{X\setminus D}\otimes L_{\blamb})).\]
But 
\[\sB^\bullet\stackrel{q.i.}{\simeq}i_{D,*}i^{-1}_DRj_*(\DR(\cM)|_{X\setminus D}\otimes L_{\blamb}).\]
Thanks to \cite[Proposition 2.2]{FSO}, we have a natural isomorphism 
\[\DR_{X\times\C^r/\C^r}(\Psi_F(\cM))\stackrel{q.i.}{\simeq}\wt i_{D,*}\wt i_D^{-1}\DR_{X\times\C^r/\C^r}(\cM(*D_F^{(r)}))\]
where $\wt i_D=(i_D,\textup{id})\colon D\times\C^r\hookrightarrow X\times\C^r$ is the associated morphism. Since $\pi_*^G$ and $\wt i_{D,*}\wt i_D^{-1}$ commute, the proof of Theorem \ref{thm:RHALX} is accomplished by Theorem \ref{thm:RHj_*} and functoriality of sheafification.

\begin{proof}[Proof of Theorem \ref{thm:RHALXhigh}]
By Lemma \ref{lm:DReq} and Theorem \ref{thm:loccohsheq}, we have 
\[\begin{aligned}
&\pi_*^G\big(\DR_{X\times\C^r/\C^r}\D R\Gamma_{[D_I\times\C^r]}\big((\D\cM)(!D^{(r)}_F)\big)\big)\\
&\simeq \DR_{X\times(\C^*)^r/(\C^*)^r}\D R\Gamma_{[D_I\times(\C^*)^r]}\big(\pi_*^G\big((\D\cM)(!D^{(r)}_F)\big)\big).
\end{aligned}
\]
By \cite[Theorem 3.11]{FSO}, we have 
\[\begin{aligned}
&\DR_{X\times(\C^*)^r/(\C^*)^r}\D R\Gamma_{[D_I\times(\C^*)^r]}\big(\pi_*^G\big((\D\cM)(!D^{(r)}_F)\big)\big)\\
&\simeq \mathbf D\DR_{X\times(\C^*)^r/(\C^*)^r} R\Gamma_{[D_I\times(\C^*)^r]}\big(\pi_*^G\big((\D\cM)(!D^{(r)}_F)\big)\big),
\end{aligned}
\]
where $\mathbf D$ denotes the duality functor for relative constructible complexes (cf. \cite[\S 2.6]{FSO}).  By Proposition \ref{prop:!*gnbpbtopoint},  $(\D\cM)(!D^{(r)}_F)$ is relative regular holonomic (cf. \cite[Definition 2.1]{FS}) and thus so is $\pi_*^G\big((\D\cM)(!D^{(r)}_F)\big)$. By regularity and Lemma \ref{lm:loccohbc}, $\DR$ and $R\Gamma$ commute and thus we have 
\[\begin{aligned}
&\mathbf D\DR_{X\times(\C^*)^r/(\C^*)^r} R\Gamma_{[D_I\times(\C^*)^r]}\big(\pi_*^G\big((\D\cM)(!D^{(r)}_F)\big)\big)\\
&\simeq \mathbf D R\Gamma_{D_I\times(\C^*)^r}\DR_{X\times(\C^*)^r/(\C^*)^r} \big(\pi_*^G\big((\D\cM)(!D^{(r)}_F)\big)\big).
\end{aligned}
\]
By \eqref{eq:dualminmax}, we know $(\D\cM)(!D^{(r)}_F)\simeq \D(\cM(*D_F^{(r)}))$ and thus
\[\begin{aligned}
&\mathbf D R\Gamma_{D_I\times(\C^*)^r}\DR_{X\times(\C^*)^r/(\C^*)^r} \big(\pi_*^G\big((\D\cM)(!D^{(r)}_F)\big)\big)\\
&\simeq\mathbf D R\Gamma_{D_I\times(\C^*)^r}\DR_{X\times(\C^*)^r/(\C^*)^r} \big(\pi_*^G\D\big(\cM(*D_F^{(r)})\big)\big)\\
&\simeq \mathbf D R\Gamma_{D_I\times(\C^*)^r} \mathbf D\DR_{X\times(\C^*)^r/(\C^*)^r} \big( \pi_*^G\big(\cM(*D_F^{(r)})\big)\big),
\end{aligned}
\]
where in the second isomorphism above follows from \eqref{eq:dualequivdirct1} in Remark \ref{rmk:reldsteqdb}(2).
By  \cite[Proposition 3.1.11]{KSbook} and the construction of $\mathbf D$, we have 
\[\mathbf D R\Gamma_{D_I\times(\C^*)^r} \mathbf D\simeq \wt i_{D_I,*}\mathbf D{\tilde i}_{D_I}^!\mathbf D\]
where $\wt i_{D_I}=(i_{D_I},\textup{id})\colon D_I\times\C^r\hookrightarrow X\times\C^r$ is the associated morphism.
Moreover, by \cite[Proposition 3.1.13]{KSbook} (see also \cite[Remark 2.24]{FSO}), for relative constructible complexes we have a functorial isomorphism
\[\mathbf D{\tilde i}_{D_I}^!\mathbf D\simeq \wt i_{D_I}^{-1}.\]
Therefore,
\[\begin{aligned}
&\mathbf D R\Gamma_{D_I\times(\C^*)^r} \mathbf D\DR_{X\times(\C^*)^r/(\C^*)^r} \big( \pi_*^G\big(\cM(*D_F^{(r)})\big)\big)\\
&\simeq \wt i_{D_I,*}\wt i_{D_I}^{-1}\DR_{X\times(\C^*)^r/(\C^*)^r} \big( \pi_*^G\big(\cM(*D_F^{(r)})\big)\big).
\end{aligned}
\]
By Theorem \ref{thm:RHj_*} we finally have
\[
\wt i_{D_I,*}\wt i_{D_I}^{-1}\DR_{X\times(\C^*)^r/(\C^*)^r} \big( \pi_*^G\big(\cM(*D_F^{(r)})\big)\big)
\simeq \wt\psi_{D_I}(\DR_X(\cM)).
\]
\end{proof}

\begin{proof}[Proof of Corollary \ref{thm:suppRH}]
By Lemma \ref{lm:relconstrnakayam} and \cite[Corollary 1]{Mai} (or Lemma \ref{lm:suppan} in Appendix \ref{sec:shffm}), we have
\[\supp_{(\C^*)^r}\psi_F(\DR\cM)=\{\blamb\in(\C^*)^r\mid \bL\wt i^*_{\blamb}(\wt\psi_F(\DR\cM))\not=0\}.\]
Since $\bL \wt i^*_{\bal}$ and $\DR$ commute, by the Riemann-Hilbert correspondence for regular holonomic $\shD_X$-modules we have 
\[\{\bal\in\C^r\mid \bL \wt i^*_{\bal}(\DR_{X\times\C^r/\C^r}\Psi_F(\cM))\not=0\}=\{\bal\in\C^r\mid \bL \wt i^*_{\bal}(\Psi_F(\cM))\not=0\}\]
By Theorem \ref{thm:RHALX}, \[G\backslash\{\bal\in\C^r\mid \bL \wt i^*_{\bal}(\DR_{X\times\C^r/\C^r}\Psi_F(\cM))\not=0\}= \{\blamb\in(\C^*)^r\mid \bL\wt i^*_{\blamb}(\wt\psi_F(\DR\cM))\not=0\}.\]
Thus, it is enough to prove 
\[\supp_{\C^r}\Psi_F(\cM)=\{\bal\in\C^r\mid \bL \wt i^*_{\bal}(\Psi_F(\cM))\not=0\}.\]
But the equality above is the analytification of \cite[Proposition 3.4.3]{BVWZ} as $\Psi_F(\cM)$ is $(n+1)$-Cohen-Macaulay. We leave details for interested readers; see also \cite[\S 3.6]{BVWZ}.
\end{proof}

\subsection{Proof of the local index formula}
In this subsection, we prove the local index formula \eqref{eq:genericindzetaf} in \S \ref{subsection:rcczetaindex}. We keep notations as in \S \ref{subsection:rcczetaindex} and \S \ref{sec:*!ext}. We focus on a small neighborhood $W\subseteq X$ around $x$.

We write $\cF^\bullet=\DR(\cM)$ for simplicity and let $q$ be the generic point of an irreducible component of $\supp_{(\C^*)^r}\psi_F\cF^\bullet$. By Corollary \ref{cor:suppam}, $\bar q$ is always a translated subtori of codimension-one. We pick a general point $\blamb\in \bar q$ and a point $\bal\in Exp^{-1}(\blamb)$. Then for some suitable $i\in\{1,\dots,r\}$ we consider the closed embedding 
\[\delta_{\hat\blamb_i}\colon \C^*\hookrightarrow (\C^*)^r,\quad t\mapsto (\lambda_1,\dots,\lambda_{i-1},t,\lambda_{i+1},\dots,\lambda_r).\]
By pullback the quasi-isomorphism in Theorem \ref{thm:RHALX} and Corollary \ref{coro:basechageSF} we obtain:
\be\label{eq:pblambdageneonq}
\wt \delta_{\hat\blamb_i}^*\big(\pi_*^G\DR_{X\times\C/\C}\big(\Psi_{f_i}(\cM_{\hat\bal_i})\big)\big)\simeq \wt\psi_{f_i}(\cF_{\hat\bal_i}),
\ee
with $\cF_{\hat\bal_i}= \DR(\cM_{\hat\bal_i})$. By the definition of $\cM_{\hat\bal_i}$, \eqref{eq:pblambdageneonq} is independent of the choices of $\bal\in Exp^{-1}(\blamb)$.

We pick a sufficient small neighborhood $V_{\bal}$ of $\bal\in\C^r$ such that $Exp(V_{\bal})\simeq V_{\bal}$ (since $Exp$ is the universal covering) and write 
\[V_{\alpha_i}=\delta_{\hat\bal_i}^{-1}(V_{\bal}).\]
We then take restriction of \eqref{eq:pblambdageneonq} on $W\times Exp(V_{\alpha_i})$. Since $\Psi_{f_i}(\cM_{\hat\bal_i})|_{W\times V_{\alpha_i}}$ is only supported on $W\times\{\alpha_i\}$, we can treat $\Psi_{f_i}(\cM_{\hat\bal_i})|_{W\times V_{\alpha_i}}$ as a regular holonomic $\shD_X$-module (over W) and hence we get 
\be\label{eq:loc1dimspnearby}
\DR_X(\Psi_{f_i}(\cM_{\hat\bal_i})|_{W\times V_{\alpha_i}})\simeq \psi_{f_i}(\cF_{\hat\bal_i})_{m_{\alpha_i}}\ee
where $m_{\alpha_i}$ is the maximal ideal of $\alpha_i\in\C$.
\begin{lemma}\label{lm:genericpbEulerind}
With notations as above, we have
\[\chi_x(\psi_{F}(\cF^\bullet),q)=\chi_x(\psi_{F}(\cF^\bullet_{\hat\bal_{i}}),m_{\alpha_i}).\]
\end{lemma}
\begin{proof}
Since the length function is additive with respect to short exact sequence, the required identity follows from the very general choice of $\bal$ in $\bar q$.
\end{proof}

We now apply \cite[Index Theorem 8.2]{Gil} to the regular holonomic $\shD_X$-module $\Psi_{f_i}(\cM_{\hat\bal_i})|_{V_{\alpha_i}}$. In consequence, the formula \eqref{eq:genericindzetaf} follows from 
\eqref{eq:loc1dimspnearby}, Lemma \ref{lm:genericpbEulerind} and Proposition \ref{prop:lambdaLalf1dim}.

\section{Appendix: relative sheafification}\label{sec:shffm}
%\section{Sheafification of sheaves of modules over commutative rings}
We discuss sheafifying sheaves of modules over commutative rings in a standard way, motivated by the study of relative holonomic $\shD$-modules over algebraic  affine spaces in \cite{Mai} and the theory of analytic relative holonomic $\shD$-modules developed in \cite{FSO,FS, FF18, FFS19}, which can be seen as the relative version of the $\sim$-functor in \cite[II.5]{HartsAG}. Compared to the references above, we mainly and more focus on studying algebraic and analytic relative supports and their differences.
\subsection{Pre-sheafification}
Let $k$ be a field, and let $X$ be topological spaces with $\sA$ a sheaf of $k$-algebra on $X$ for some base field $k$. For a commutative $k$-algebra $R$, we obtain a new sheaf of rings
\[\sA_R=\sA\otimes_k R.\]
We write by $\Mod(\sA_R)$ the abelian category of $\sA_R$-modules on $X$.\footnote{Throughout this paper, all the modules over non-commutative rings are assumed to be the left ones, unless indicated otherwise.} For $\sM\in \Mod(\sA_R)$, we define a presheaf on $X\times \Spec R$ (use the product topology) by 
\[\widetilde\sM^{\textup{pre}}(U\times \Spec R_f)=\sM(U)\otimes_R R_f\]
where $R_f$ is the localization of $R$ with respect to $f\in R$. Since $\Spec R_f$ give a basis of the Zariski topology of $\Spec R$, the presheaf above gives a sheaf on $X\times \Spec R$, denoted by $\widetilde\sM^{\textup{pre}}$. We call such procedure the \emph{relative pre-sheafification} of $\sM$.
Consequently, $\widetilde\sM^{\textup{pre}}$ is an $\wt\sA_R^{\textup{pre}}$-module on $X\times \Spec R$. One can see that when $X$ is a point, the relative pre-sheafification is the $\sim$-functor in algebraic geometry transforming $R$-modules into quasi-coherent sheaves on $\Spec R$.

\begin{lemma}\label{lm:stalkpre}
The stalk of $\wt\sM^{\textup{pre}}$ at $(x,m)\in X\times \Spec R$ satisfies 
\[\wt\sM^{\textup{pre}}_{(x,m)}=\sM_{x}\otimes_R R_m=(\sM\otimes_R R_m)_x\]
where $\sM_x$ is the stalk of $\sM$ at $x$ and similarly for $(\sM\otimes_R R_m)_x$.
\end{lemma}
\begin{proof}
Since localizations are colimits, we apply \cite[Part 1, Lemma 4.14.10]{stacks-project}.
\end{proof}

Let $I\subseteq R$ be an ideal. The \emph{relative pullback} of $\sM$ to $\Spec R/I$ is defined as 
\[\tilde i^*_{\textup{pre}}(\wt\sM^{\textup{pre}})=\wt{\sM\otimes_R R/I}^{\textup{pre}}\]
where $\tilde i\colon X\times \Spec R/I\hookrightarrow X\times \Spec R$ is the closed embedding induced by $R\to R/I$. More generally, if $h\colon \Spec S\to \Spec R$ is a morphism of affine schemes, then one can define the \emph{relative pullback functor} $\tilde h^*_{\textup{pre}}$ similarly. If $m\subseteq R$ is a maximal ideal, the pullback of $\sM$ at the closed point of $m$ gives a $\sA$-module on $X$. Hence, $\wt\sM$ is also called a relative $\sA$-module over $\Spec R$.

Let $\mathfrak p \subseteq R$ be a prime ideal. Then the \emph{specialization} of $\sM$ at $\mathfrak p$ is 
\[\sM_{\mfp}=\sM\otimes_R R_{\mfp}\]
where $R_{\mfp}$ is the localization of $R$ at $\mfp$, which is a $\sA_{R_{\mfp}}$-module on $X$. The \emph{complete specialization} of $\sM$ at $p$ is
\[\hat\sM_{\mfp}=\sM\otimes_R \hat R_{\mfp}\]
where $\hat R_{\mfp}$ is the completion of $R_{\mfp}$ with respect to $\mfp$. Then $\hat\sM_{\mfp}$ is an $\hat\sA_{R_{\mfp}}$-module.
\begin{definition}\label{def:relsupp}
For $\sM\in \Mod(\sA_R)$, the relative support of $\sM$ is 
\[\supp_{\Spec R}\sM=\{\textup{maximal ideals }m\in \Spec R\mid \sM_m\not=0\}.\]
More generally, if $\sM^\bullet\in D^b(\sA_R)$, the bounded derived category of $\Mod(\sA_R)$, then we define
\[\supp_{\Spec R}\sM^\bullet=\bigcup_i \supp_{\Spec R}\cH^i(\sM^\bullet).\]
\end{definition}

\subsection{Analytic sheafification}\label{subsec:ansh}
We now assume $X$ a complex manifold, $\sA$ a sheaf of $\C$-algebra and $R$ a commutative $\C$-algebra of finite type. We write $\Specan R$ the analytic scheme of $\Spec R$.

For $\sM\in \Mod(\sA_R)$, the analytic sheafification of $\sM$ is defined as 
\[\wt\sM=\iota^{-1}(\wt\sM^{\textup{pre}})\otimes_{\wt{pr}^{-1}_X\sO_{\Spec R}} pr_X^{-1}\sO_{\Specan R},\]
where $\iota: X\times \Specan R\to X\times \Spec R$ is the natural inclusion map induced by $\iota_R\colon \Specan R\to \Spec R$ from GAGA principle, $pr_X\colon X\times \Specan R\to \Specan R$ is the projection and $\wt{pr}_X$ is the composition of $pr_X$ and $\iota_R$. The output $\wt\sM$ gives a $\wt\sA_R$-module on $X\times\Specan R$. 

If $X$ is an algebraic scheme over $k$ and $R$ a commutative $k$-algebra, then we can define the algebraic sheafification of $\sM\in \Mod(\sA_R)$ by 
\[\wt\sM^{\textup{alg}}=\iota^{-1}_{\textup{alg}}(\wt\sM^{\textup{pre}})\]
where $\iota^{-1}_{\textup{alg}}: X_R\to X\times\Spec R$ is the natural inclusion map with $X_R$ the scheme of $X$ over $R$. 

\subsection{Analytic sheafification for $\shD_{X,R}$-module and relative holonomicity}\label{subsec:ashfdmo}
Now we assume $\sA=\shD_X$, the sheaf of holomorphic differential operators on a complex manifold $X$, and $R$ a commutative finite generated $\C$-algebra. In this case, we also want the analytic sheafification to be consistent with the analytic structure sheaf of $X\times \Specan R$ (not only with that of the second factor). For instance, the sheafification of an $\sO_{X,R}$-module should be its GAGA analytification. 
Thus, for $\sM\in \Mod(\shD_{X,R})$ (or $\sM\in\Mod(\sO_{X,R}$)) we define 
\[\wt\sM=\iota^{-1}(\wt\sM^{\textup{pre}})\otimes_{\iota^{-1}(\wt\sO_{X,R}^{\textup{pre}})} \sO_{X\times \Specan R}.\]
%where $pr\colon X\times \Specan R\to X$ is the projection. 
Therefore, the output $\wt\sM$ in this case gives a $\wt\shD_{X,R}=\shD_{X\times\Specan R/\Specan R}$-module, i.e. a relative analytic $\shD$-module. Let us refer to \cite[Chapter III. 1.3]{Schpbook} and \cite[\S 3]{FSO} for the general theory of relative analytic $\shD$-modules; see also \cite[\S 2]{Wuch}). The relative pullback for $\wt\sM$ in this case is exactly the base change for relative $\shD$-modules \cite[\S 2.2]{Wuch}. 

\begin{remark}\label{rmk:algshd}
When $X$ is an smooth complex algebraic variety and $\sM$ is an algebraic $\shD_{X,R}$-module, we define the algebraic sheafification by
\[\wt\sM^{\textup{alg}}=\iota^{-1}_{\textup{alg}}(\wt\sM^{\textup{pre}})\otimes_{\iota^{-1}(\wt\sO_{X,R}^{\textup{pre}})} \sO_{X_R},\]
which gives an algebraic relative $\shD$-module over $\Spec R$. In this case, the functor $\sim_{\textup{alg}}$ has a quasi-inverse $p_*$: 
\[p_*(\wt\sM^{\textup{alg}})\simeq \sM,\]
where $p\colon X_R\to X$ is the natural projection. 
\end{remark}

\begin{definition}\label{def:relcohhol}
For $\sM\in \Mod(\shD_{X,R})$, we say $\sM$ relative coherent over $\Specan R$ if 
%\begin{enumerate}
    %\item $\sM_m$ is coherent over $\shD_{X,R_m}$ for every maximal ideal $m\in\Spec R$, 
    %\item 
    $\wt\sM$ is coherent over $\wt\shD_{X,R}$.
    %so over $\shD_{X\times\Specan R/\Specan R}$ on $X\times \Specan R$ (cf. \cite[Definition 2.4]{Wuch}\footnote{The relative coherence and holonomicity in \emph{loc. cit.} are only defined over smooth space. But the definition can be easily generalized over $\Specan R$.}
%\end{enumerate}
We say $\sM$ $($or $\wt\sM\in\Mod(\wt\shD_{X,R}))$ relative holonomic if it is relative coherent and locally on a relatively compact open subset $W$ of $X$
\[\Ch^{\rel}(\wt\sM|_{W\times \Specan R})\subseteq \Lambda\times \Specan R,\]
where $\Lambda$ is a $($possibly reducible$)$ conic Lagrangian subvariety inside the cotangent bundle $T^*X$ $($over W$)$ and $\Ch^{\rel}(\wt\sM)$ is the relative characteristic variety.% $($see also \cite[Definition 2.4]{Wuch} for a more general definition of relative holonomicity$)$.%\footnote{The relative characteristic variety in \emph{loc. cit.} is only defined over smooth space. But the definition can be easily generalized over $\Specan R$.}
\end{definition}
%and \cite[Definition 3.2.3]{BVWZ} for the (local) algebraic case). 
Notice that coherence over $\shD_{X,R}$ on $X$ implies relative coherence, but not vice versa. The relative holonomicity above is the analytic sheafification of the algebraic relative holonomicity \cite[Definition 3.2.3]{BVWZ}, which in turn follows the analytic relative holonomicity defined in \cite[\S 3.4]{FSO}. By definition, the category of relative holonomic modules is abelian. 

The following lemma is well-known; see for instance \cite[Proposition 8]{Mai}, \cite[Lemma 2.10]{FF18}\footnote{Lemma 2.10 in \cite{FF18} seems to be not completely correct since locally on $X$ the index set $I$ can be infinite by Theorem \ref{thm:relccgny}.} and also \cite[Proposition 3.2.5]{BVWZ}. 
\begin{lemma}\label{lm:altrelhol}
If $\wt\sM\in\Mod(\wt\shD_{X,D})$ is relative holonomic, then for every pair $(W,V)$ with  $W$ a relatively compact open subset of $X$ and $V$ a relatively compact open subset of $\Specan R$, there exist a finite number of closed analytic subvarieties $S_w\subseteq V$ such that 
\[\Ch^\rel(\wt\sM|_{W\times V})=\bigcup_w \Lambda_w\times S_w,\]
where $\Lambda_w$ are irreducible conic Lagrangian subvarieties in $T^*X|_W$. If moreover $\sM\in \Mod(\shD_{X,R})$ is coherent over $\shD_{X,R}$ and relative holonomic over $\Specan R$, then \[\Ch^\rel(\wt\sM|_{W\times \Specan R})=\bigcup_w \Lambda_w\times S_w\]
with a finite number of closed analytic subvarieties $S_w\subseteq \Specan R$.
\end{lemma}
\begin{proof}
For completeness, we present here a proof essentially due to Maisonobe \cite{Mai}. We write 
\[\Ch^\rel(\wt\sM|_{W\times\Specan R})\subseteq \Lambda\times \Specan R\]
with $\Lambda=\sum_w \Lambda_w$ a finite union of irreducible conic Lagrangian subvarieties in $T^*X|_W$. 
Notice that the inclusion above does not imply that $\Ch^\rel(\wt\sM|_{W\times\Specan R})$ has finite many irreducible components. 
For every $p\in V$, by the relative Bernstein inequality (see \cite[Proposition 5]{Mai} and also \cite[Theorem 2.2]{Wuch}), we have 
\[\Ch^\rel(\wt\sM)_p\coloneqq\Ch^\rel(\wt\sM)\cap pr^{-1}(p)\subset \sum_w \Lambda_w,\]
where $pr\colon T^*X\times \Specan R\to \Specan R$ is the second projection. Then we define 
$$S_w=\{p\in V \mid \Lambda_w\subseteq \Ch^\rel(\wt\sM)_p \},$$
from which we obtain 
\[\Ch^\rel(\wt\sM|_{W\times V})=\bigcup_w \Lambda_w\times S_w.\]
The set $S_w$ is analytic and closed since 
\[\{\lambda_w\}\times S_w=\Ch^\rel(\wt\sM|_{W\times V})\cap pr_1^{-1}(\lambda_w)\]
with $\lambda_w$ a general point on $\Lambda_w$, where $pr_1\colon T^*X\times \Specan R\to T^*X$ is the first projection.

If additionally $\sM$ is coherent over $\shD_{X,R}$, we can define 
$$S_w=\{p\in \Specan R \mid \Lambda_w\subseteq \Ch^\rel(\wt\sM)_p \}.$$
Since $\Ch^\rel(\wt\sM|_{W\times\Specan R})$ is a closed analytic subvariety of $T^*W\times \Specan R$, we conclude that $S_w\subseteq \Specan R$ is analytic and closed.  \end{proof}

For a maximal ideal $m\in \Spec R$, we write $p_m=\iota^{-1}_R(m)\in \Specan R$. For $(x,p)\in X\times\Specan R$, $\wt\sM_{(x,p)}$ denotes the (analytic) localization of $\wt\sM$ at $(x,p)$.
\begin{lemma}\label{lm:suppan}
If $\sM$ is relative coherent over $\Specan R$, then 
\[\supp_{\Spec R}\sM=\{m\in \Spec R\mid \wt\sM_{(x,p_m)}\not=0 \textup{ for some $x\in X$}\}.\]
%where $$\eta_m\colon X\times \Specan R_m \hookrightarrow X\times \Specan R$$ is the natural embedding with $\Specan R_m$ the analytic germ.
\end{lemma}

\begin{proof}
We first assume $R=\C[{\bf y}]$ for ${\bf y}=(y_1,\dots,y_l)$ for some integral $l>0$. Since completion of local rings is faithfully flat, we have
\[\hat\sM_m\not=0 \Leftrightarrow \sM_m\not=0.\]
By definition, we have
\[\hat\sM_m\not=0 \Leftrightarrow (\hat\sM_m)_x \textup{ for some }x\in X,\]
where $(\hat\sM_m)_x$ is the localization of $\hat\sM_m$ at $x$ because $\hat\sM_m$ is a sheaf on $X$. We write by $\C\{{\bf x}\}$ the stalk of $\sO_X$ at $x$, and by $\C[[{\bf x}]]$ the ring of formal power series. Since $\C[[{\bf x}]]$ is faithfully flat over $\C\{{\bf x}\}$, we know 
\[(\hat\sM_m)_x\not=0 \Leftrightarrow (\hat\sM_m)_x\otimes_{\C\{{\bf x}\}}\C[[{\bf x}]] \not=0.\]
We also have 
\[(\hat\sM_m)_x\otimes_{\C\{{\bf x}\}}\C[[{\bf x}]]\simeq \sM_x\otimes_{\C\{{\bf x}\}\otimes_\C R} \C[[{\bf x}]][[{\bf y}]].\]
Thanks to the fact that the completion of local rings is faithfully flat again, we have 
\[\wt\sM_{(x,p_m)}\not=0 \Leftrightarrow \wt\sM_{(x,p_m)}\otimes_{\C\{{\bf x,y}\}}\C[[{\bf x,y}]] \not=0.\]
By Lemma \ref{lm:stalkpre}, we have 
\[\sM_x\otimes_{\C\{{\bf x}\}\otimes_\C R} \C[[{\bf x}]][[{\bf y}]]\simeq \wt\sM_{(x,p_m)}\otimes_{\C\{{\bf x,y}\}}\C[[{\bf x,y}]].\]
Therefore, we conclude 
\[\sM_m\not=0 \Leftrightarrow\wt\sM_{(x,p_m)}\not=0 \textup{ for some }x\in X.\]
In general, considering a surjection 
\[\C[{\bf y}]\twoheadrightarrow R,\]
we are done by similar arguments. 
\end{proof}
%We write 
%$$\supp_{\Specan R}\sM=\iota_R^{-1}(\supp_{\Spec R}\sM)$$ for $\sM\in \Mod(\shD_{X,R})$.
\begin{prop-def}\label{prop:suppprch}\label{prop-def:relsupp}
Let $\sM\in \Mod(\shD_{X,R})$ be relative coherent over $\Specan R$. 
Then 
\[\supp_{\Specan R}\sM\coloneqq\iota_R^{-1}(\supp_{\Spec R}\sM)=\supp_{\Specan R}\wt\sM\coloneqq pr(\Ch^\rel(\wt\sM)),\]
where $pr\colon T^*X\times \Specan R\to \Specan R$ is the projection. 
\end{prop-def}
\begin{proof}
We focus ourselves on the characteristic variety
\[\Ch^\rel(\wt\sM)\subseteq T^*X\times \Specan R\]
locally over $W\times V$ by picking a good filtration $F_\bullet \wt\sM$. Then $\Ch^\rel(\wt\sM)$ is the support of $\wt\gr^F_\bullet\wt\sM$ inside $T^*X\times \Specan V$, which means $\Ch^\rel(\wt\sM)$ is a closed analytic subvariety over $W\times V$. We observe by coherence
\[\wt\sM_{(x,p)}\not=0 \Leftrightarrow (\wt\gr^F_\bullet\wt\sM)_{(x,\xi,p)}\not=0 \textup{ for some $\xi\in T^*_xX$}.\]
Thus, we have 
\[\wt\sM_{(x,p)}\not=0 \Leftrightarrow (x,\xi,p)\in \Ch^\rel(\wt\sM) \textup{ for some $\xi\in T^*_xX$}.\]
By Lemma \ref{lm:suppan}, we hence conclude 
\[\iota_R^{-1}(\supp_{\Spec R}\sM)=pr(\Ch^\rel(\wt\sM))\]
locally over $W\times V$. 
\end{proof}

\subsection{Bernstein-Sato ideal}\label{def:bsideal}
If $\sM\in\Mod(\shD_{X,R})$ is coherent over $\shD_{X,R}$, then we define 
\[B(\sM)=\Ann_{R}(\sM)\subseteq R,\]
calling it the Bernstein-Sato ideal of $\sM$. We write by $Z(B(\sM))\subseteq \Spec R$ the zero locus of $B(\sM)$. A priori, since $\sM$ is not finite generated over $R$ in general, we only have 
$$\supp_{\Spec R}\sM\subseteq Z(B(\sM)),$$
but we do not know if $\supp_{\Spec R}\sM$ is an algebraic subvariety.

\begin{prop}\label{prop:zbfsupp=}
If $\sM\in \Mod(\shD_{X,R})$ is coherent over $\shD_{X,R}$ and relative holonomic over $\Specan R$, then locally on a relatively compact open subset $W\subseteq X$ we have 
\[\supp_{\Spec R}\sM|_W= Z(B(\sM|_W)).\]
\end{prop}
\begin{proof}
For simplicity, we assume $X=W$.
By \cite[Proposition 9]{Mai} (see also \cite[Lemma 3.4.1]{BVWZ}), we have 
\[Z(B(\sM))=pr(\Ch^\rel(\wt\sM)).\]
Then the required equality follows from Proposition-Definition \ref{prop-def:relsupp}.
\end{proof}

By the proposition above and Proposition \ref{lm:altrelhol}, we have:
\begin{coro}\label{cor:algdecompofchrel}
If $\sM\in \Mod(\shD_{X,R})$ is coherent over $\shD_{X,R}$ and relative holonomic over $\Specan R$, then over a relatively compact open subset $W\subseteq X$ \[\Ch^\rel(\wt\sM|_{W\times \Specan R})=\bigcup_w \Lambda_w\times S_w\]
with a finite number of closed algebraic subvarieties $S_w\subseteq \Spec R$.
\end{coro}

Applying Proposition-Definition \ref{prop-def:relsupp}, Lemma \ref{lm:altrelhol} and Proposition \ref{prop:zbfsupp=}, we have: 
%\begin{remark}
%Condition (1) in Definition \ref{def:relcohhol} is not used in the proof of both Lemma \ref{lm:suppan} and Proposition \ref{prop:suppprch}. Therefore, Lemma \ref{lm:suppan} and Proposition \ref{prop:suppprch} hold more generally for $\sM\in \Mod(\shD_{X,R})$ satisfying $\wt\sM$ coherent over $\wt\shD_{X,R}$. 
%\end{remark}
\begin{coro}\label{cor:suppanalgsubvar}
Let $\sM\in \Mod(\shD_{X,R})$ be relative holonomic over $\Specan R$. Then
\begin{enumerate}
   \item $\supp_{\Specan R}\sM$ is analytically closed locally over both $X$ and $\Specan R$, i.e. over each open subset $W\times V$ satisfying $W\subseteq X$ a relatively compact open subset and $V\subseteq \Specan R$ a relatively compact open subset, 
    \[\supp_{\Specan R}(\sM|_W) \cap V\subseteq \Specan R \]
   is a closed analytic subvariety.
   \item if $\sM$ is also coherent over $\shD_{X,R}$, then  
$\supp_{\Spec R}\sM$ is algebraically closed locally over $X$, i.e. over a relatively compact open subset $W\subseteq X$ 
\[\supp_{\Spec R}(\sM|_W) \subseteq \Spec R \]
is a closed algebraic subvariety.
\end{enumerate}
\end{coro}

\subsection{Sheafification for derived category}
By construction, the analytic sheafification is a faithful and exact functor from $\Mod(\sA_R)$ to $\Mod(\wt\sA_R)$ for $\sA$ in general (similarly for the algebraic sheafification). We define $\wt\Mod(\sA_R)$ as the subabelian category of $\Mod(\wt\sA_R)$ as the image of $\Mod(\wt\sA_R)$ under the analytic sheafification. We write by $\wt{D^b}(\sA_R)$ the derived category of $\wt\Mod(\sA_R)$ and by $D^b(\wt\sA_R)$ the derived category of $\Mod(\wt\sA_R)$. %It is natural to ask a GAGA type question: \[\textup{Under what conditions, $D^b(\wt{\sA}_R)$ and $\wt{D^b}(\sA_R)$ are equivalent?}\]
We can define the \emph{analytic relative pullback} functor $\wt h^*$ for an affine morphism $h:\Spec S\to\Spec R$ in an obvious way. We denote by $\bL \wt h^*$ its derived functor. By slightly abuse of notations, we use $\wt h^*$ to denote the functor for both  $\wt\Mod(\sA_R)$ and $\Mod(\wt\sA_R)$
and similarly for $\bL\wt h^*$ in the derived case.

For $\wt\shD_{X,R}$-module, if $\mu\colon T\to \Specan R$ is a morphism between analytic schemes, we use $\wt u^*$ (resp. $\bL\wt u^*$) to denote the (resp. derived) pullback functor of the natural morphism $\wt\mu\colon X\times T\to X\times \Specan R$ induced by base-change (see also \cite[\S2.2]{Wuch}). 

\subsection{Sheaves of algebraic local cohomology} Let $X$ be a complex manifold and let $R$ be a commutative $\C$-algebra of finite type. 
For $Z\subseteq X$ an analytic subvariety, we use $R\Gamma_{[Z]}$ to denote the right derived functor of sheaves of algebraic local cohomology along $Z$ (cf. \cite[Chapter II.5]{Bj}). We denote by $\wt Z=Z\times \Specan R$. For $\sM\in \Mod(\shD_{X,R})$, by construction we have a natural isomorphism
\[\wt{R\Gamma}_{[Z]}(\sM)\simeq R\Gamma_{[\wt Z]}(\wt\sM).\]

The following lemma is standard and we leave its proof for interested readers. 
\begin{lemma}\label{lm:loccohbc}
Let $h\colon \Spec S\to\Spec R$ be an affine morphism. Then for $\wt \sM\in \Mod(\wt \shD_{X,R})$ we have a natural isomorphism 
\[\bL \wt h^*R\Gamma_{[\wt Z]}(\wt\sM)\simeq R\Gamma_{[Z\times \Specan S]}(\bL \wt h^*\wt\sM).\]
\end{lemma}

\subsection{Duality for relative coherent $\shD$-module}
We keep notation as in \S\ref{subsec:ashfdmo}. We assume $\dim_\C X=n$. For $\sM\in\Mod(\shD_{X,R})$, if $\sM$ is coherent over $\shD_{X,R}$, then we define
\[\D_R(\sM)=\Rhom_{\shD_{X,R}}(\sM,\shD_{X,R})\otimes_\sO \omega^{-1}_{X}[n]\in D^b(\shD_{X,R}),\]
where $\omega_X$ is the sheaf of holomorphic $n$-forms on $X$.
If $\wt\sM$ is relative coherent over $\Specan R$, then we define
\[\D_{\Specan R}(\wt\sM)=\Rhom_{\wt\shD_{X,R}}(\wt\sM,\wt\shD_{X,R})\otimes_{\wt\sO} \wt{\omega^{-1}_X\otimes_\C R}[n]\in D^b(\wt\shD_{X,R}).\]
Notice that by construction, 
\[\wt{\omega^{-1}_X\otimes_\C R}\simeq pr_X^*\omega^{-1}_X,\]
where $pr_X\colon X\times\Specan R\to X$ is the projection. When $R=\C$, $\D_\C(\bullet)$ is just the usual duality functor for coherent $\shD_X$-modules. When $R$ (resp. $\Specan R$) is obvious from the context, we use $\D$ to denote $\D_R$ (resp. $\D_{\Specan R}$) for short. 

By construction, we immediately obtain the following lemma, with its proof skipped.
\begin{lemma}\label{lm:basechangeanddual}
Let $h: \Spec S\to \Spec R$ be a finite type morphism of affine schemes. For $\sM\in\Mod(\shD_{X,R})$, we have:
\begin{enumerate}
    \item if $\sM$ is coherent over $\shD_{X,R}$, then 
    \[\D_S(\sM\otimes^\bL_R S)\simeq \D_R(\sM)\otimes^\bL_R S,\]
    \item if $\sM$ is relative coherent over $\Specan R$, then 
    \[\D_{\Specan S}(\bL \wt h^*\wt\sM)\simeq \bL\wt h^*\D_{\Specan R}(\wt\sM).\]
\end{enumerate}
\end{lemma}

The following definition is the analytification of \cite[Definition 3.3.1 and 4.3.4]{BVWZ}.
\begin{definition}
For  $\wt\sM\in \Mod(\wt\shD_{X,R})$ relative coherent over $\Specan R$, we say $\wt\sM$ is $j$-Cohen-Macaulay for some $j\in\Z_{\ge 0}$ if  $\wt\sM\not=0$ and
\[\Ext^k_{\wt\shD_{X,R}}(\wt\sM,\wt\shD_{X,R})=0\textup{ for every $k\not= j$}.\]
If $\wt\sM$ is $n$-Cohen-Macaulay, then for simplicity we also use $\D(\wt\sM)$ to denote 
\[\Ext^n_{\wt\shD_{X,R}}(\wt\sM,\wt\shD_{X,R})\otimes_\sO \wt{\omega^{-1}_X\otimes_\C R}.\]
\end{definition}
%whenever $\sM|_{W\times V}\not=0$, on every $W\times V$ with $W\subseteq X$ a relatively compact open subset and $V\subseteq \Specan R$ a relatively compact open subset.

\subsection{The relative codimension filtration of Gabber-Kashiwara}\label{subsec:relcodimfil}
In this subsection, we let $R$ be a regular commutative finitely generated $\C$-algebra integral domain.
The following lemma is well-known (see for instance \cite[Lemma 2.9]{FF18} and \cite[\S 3.6]{BVWZ}).
\begin{lemma}\label{lm:angrnchreq}
Let $\wt\sM$ be relative coherent over $\Specan R$. If $\wt\sM_{(x,p)}\not=0$ for some point $(x,p)\in X\times\Specan R$, then 
\[j(\wt\sM_{(x,p)})+\dim_\C(\Ch^{\rel}(\wt\sM)\cap \pi_R^{-1}(\wt W))=2n+\dim R,\]
where $\pi_R\colon T^*X\times\Specan R\to X\times \Specan R$ is the natural projection, $\wt W$ is a small open neighborhood of $(x,p)$ and $j(\bullet)$ denotes the grade number for $(\shD_{X,R})_{(x,p)}$-modules $($see for instance \cite[Definition 4.3.1]{BVWZ}$)$.
\end{lemma}

Following the recipe in \cite[Gabber-Kashiwara theorem]{Gil} and \cite[\S 2.4]{Kasbook}, we now construct the Gabber-Kashiwara codimension filtration under the relative setting. For $\sM$ a coherent $\shD_{X,R}$-module and $k\in\Z_{\ge0}$, we define a submodule
\[T_k(\sM)=\{m\in \sM\mid \textup{codim}_\C\Ch^\rel(\wt\shD_{X,R}\cdot m)\ge k\}\subseteq \sM,\]
where $\wt\shD_{X,R}\cdot m\subseteq\wt\sM$ is the submodule generated by $m$. 
Then we have a finite decreasing filtration, called the \emph{relative codimension filtration},
\[\cdots \subseteq T_k(\sM)\subseteq T_{k-1}(\sM)\cdots \subseteq T_1(\sM)\subseteq T_0(\sM)=\sM.\]
In fact, by Theorem \ref{thm:relgabberkashiwara} and the relative Bernstein inequality \cite[Theorem 2.2]{Wuch}, the relative codimension filtration always stops at $k=n+\dim R$.

If $\sM$ is only relative coherent over $\Specan R$ (or more generally $\wt\sM$ is a coherent $\wt\shD_{X,R}$-module), we set \[T_k(\wt\sM)=\{m\in \wt\sM\mid \textup{codim}_\C\Ch^\rel(\wt\shD_{X,R}\cdot m)\ge k\}\subseteq \wt\sM.\]
One can easily check that the relative codimension filtration is compatible with analytic sheafification, that is, if $\sM$ is coherent over $\shD_{X,R}$, then 
\[\wt{T_k(\sM)}=T_k(\wt\sM).\]

\begin{definition}\label{def:purityofcodim}
For a non-zero coherent $\shD_{X,R}$-moudle $\sM$, we say $\sM$ pure of codimension $l$ if 
\begin{enumerate}
    \item $T_l(\sM)=\sM$ and 
    \item $T_k(\sM)=0$ if $k<l$.
\end{enumerate}
For a coherent $\wt\shD_{X,R}$-module $\wt\sM$, we define its purity similarly. 
\end{definition}
By definition, a coherent submodule of a pure module is always pure of the same codimension.  

We have the following relative Gabber-Kashiwara theorem:
\begin{theorem}\label{thm:relgabberkashiwara}
Let $\wt\sM$ $($or $\sM)$ be a coherent $\wt\shD_{X,R}$-module $($resp. $\shD_{X,R}$-module$)$. Then we have:
\begin{enumerate}
    \item $T_k(\wt\sM)$ $($or $T_k(\sM))$ is coherent over $\wt\shD_{X,R}$ $($resp. $\shD_{X,R})$.
    \item $\Ch^{\rel}\frac{T_{k}(\wt\sM)}{T_{k+1}(\wt\sM)}$ is purely $k$-codimensional $($if not empty$)$, i.e. every irreducible component is of codimension $k$.
    \item if $\frac{T_{k}(\wt\sM)}{T_{k+1}(\wt\sM)}\not=0$ $($or $\frac{T_{k}(\sM)}{T_{k+1}(\sM)}\not=0)$, then it is pure of codimension $k$. 
\end{enumerate}
\end{theorem}
\begin{proof}
We follow the strategy in \cite[2.4]{Kasbook}.
We use Lemma \ref{lm:angrnchreq} to replace the role of \cite[Theorem 2.19]{Kasbook}. Then the proof of Theorem 2.18 and Theorem 2.24 in \emph{loc. cit.} give us all the required statements. If $\sM$ is coherent over $\shD_{X,R}$, then we run the arguments in \emph{loc. cit.} on $X$ replacing $\shD_X$ by $\shD_{X,R}$. If $\wt\sM$ is coherent over $\wt\shD_{X,R}$, we run the same arguments but replacing $X$ by $X\times\Specan R$ and $\shD_X$ by $\wt\shD_{X,R}$.
\end{proof}

By Lemma \ref{lm:angrnchreq} and Theorem \ref{thm:relgabberkashiwara}, we immediately obtain:
\begin{coro}\label{cor:puritylocgcomp}
If $\wt\sM$ is pure of codimension $l$, then for $(x,p)\in X\times \Specan R$ if $\wt\sM_{(x,p)}\not=0$ then $\wt\sM_{(x,p)}$ is $l$-pure over $(\wt\shD_{X,R})_{(x,p)}$ in the sense of \cite[Definition 4.3.4]{BVWZ}.
\end{coro}

If $\sM$ is relative coherent over $\Specan R$, we set 
\[S_k(\sM)=\supp_{\Specan R}\frac{T_{k}(\wt\sM)}{T_{k+1}(\wt\sM)}.\]
Then by Corollary \ref{cor:suppanalgsubvar} and Theorem \ref{thm:relgabberkashiwara} (2), we have:
\begin{coro}\label{cor:pureS_km}
We have:
\begin{enumerate}%[label=(\roman*)]
    \item If $\sM$ is relative holonomic over $\Specan R$, then over each open subset $W\times V$ satisfying $W\subseteq X$ a relatively compact open subset and $V\subseteq \Specan R$ a relatively compact open subset, $S_k(\sM)\subseteq V$ is a closed analytic subvariety and if $S_k(\sM)\not=\emptyset$ then it is  purely $(k-n)$-codimensional.
    \item If $\sM$ is coherent over $\shD_{X,R}$ and relative holonomic over $\Specan R$, then over each relatively compact open subset $W\subseteq X$, 
    $S_k(\sM)\subseteq \Spec R$ is a closed algebraic subvariety and if 
    $S_k(\sM)\not=\emptyset$ then it is  purely $(k-n)$-codimensional.
\end{enumerate}
\end{coro}

The following proposition links Cohen-Macaulay modules to pure modules.
\begin{prop}\label{prop:suppsubcmrh}
If $\wt\sM$ (or $\sM$) is relative holonomic over $\Specan R$ and $(n+j)$-Cohen-Macaulay for some $j$, then every non-zero coherent submodule of $\wt\sM$ is pure of codimension $j$. In particular, locally on every open subset $W\times V$ satisfying $W$ a relatively compact open subset of $X$ and $V$ a relatively compact open subset of $\Specan R$, $$\supp_{\Specan R}\wt\sM|_{W\times V}\subseteq V$$
is purely $j$-codimensional. If moreover $\sM$ is coherent over $\shD_{X,R}$, then locally over every relatively compact open subset $W\subseteq X$, $\supp_{\Spec R}\sM\subseteq \Spec R$ is purely $j$-codimensional.
\end{prop}
\begin{proof}
We only need to prove $T_k(\wt\sM)=0$ for every $k>n+j$. We assume on the contrary,  $T_k(\wt\sM)_{(x,p)}\not=0$ for some $k>n+j$ and some $(x,p)\in X\times \Specan R$. For a $(\wt\shD_{X,R})_{(x,p)}$-module $M$, for simplicity we write
\[E^l(M)=\textup{Ext}^l_{(\wt\shD_{X,R})_{(x,p)}}(M,(\wt\shD_{X,R})_{(x,p)}).\]
By Lemma \ref{lm:angrnchreq}, the grade number
\[j(T_k(\wt\sM)_{(x,p)})\ge k>n+j.\]
Then $E^{n+j}(T_k(\wt\sM)_{(x,p)})=0$ and hence $E^{n+j}(E^{n+j}(T_k(\wt\sM)_{(x,p)}))=0$.
Taking double-dual of the inclusion 
\[T_k(\wt\sM)_{(x,p)}\hookrightarrow \wt\sM_{(x,p)}\]
by Cohen-Macaulayness and naturality we obtain a commutative diagram
\[
\begin{tikzcd}
T_k(\wt\sM)_{(x,p)}\arrow[r,hook]\arrow[d]&  \wt\sM_{(x,p)} \arrow[d,"\simeq"]\\
E^{n+j}(E^{n+j}(T_k(\wt\sM)_{(x,p)}))\arrow[r]&  E^{n+j}(E^{n+j}(\wt\sM_{(x,p)})).
\end{tikzcd}
\]
We then end up with a contradiction since $E^{n+j}(E^{n+j}(T_k(\wt\sM)_{(x,p)}))=0.
$

By purity, $S_{n+j}(\wt\sM)=\supp_{\Specan r}(\wt\sM)$. The other required statements follow from Corollary \ref{cor:pureS_km}.
\end{proof}

%\begin{remark}
%In the situation of Proposition \ref{prop:suppsubcmrh}, using the arguments in proving \cite[Theorem 2.18]{Kasbook}, one can obtain the pure codimension filtration for coherent $\wt\shD_{X,R}$-modules, i.e. the Gabber-Kashiwara Theorem in the relative setting. Using the pure codimension filtration for Cohen-Macaulay relative holonomic modules, one can indeed prove that
%$\supp_{\Specan R}\wt\sN$ is purely $(\dim R-j)$-dimensional, i.e. each irreducible component of $\supp_{\Specan R}\wt\sN$ is $(\dim R-j)$-dimensional.
%\end{remark}

\subsection{A relative Nakayama lemma}\label{subsec:anshcon}
In this subsection, our goal is to prove  a relative Nakayama lemma for $\C$-constructible sheaves of $R$-modules (cf. \cite[Definition 8.5.6]{KSbook}). 

We assume $X$ a complex manifold, $\sA=\C_X$ the constant sheaf on $X$, and $R$ a noetherian commutative $\C$-algebra.  We write the derived category of $\C$-constructible (or constructible for short) sheaves of $\C$-vector spaces by
$D^b_{c}(\C_X)$, by
$D^b_{c}(R_X)$ the derived category of $\C$-constructible sheaves of $R$-modules and by $\wt D^b_{c}(R_X)$ the derived category of the analytic sheafification of $\C$-constructible sheaves of $R$-modules. We say a (bounded) complex is 0 in the derived category if it is quasi-isomorphic to the 0 complex in the derived category.

One can check that when $R$ is a commutative $\C$-algebra of finite type, if 
$$\cF^\bullet\in D^b_c(R_X),$$ then $\wt\cF^\bullet$ is a relative constructible complex over $\C^l$ supported on $X\times \Specan R$ in the sense of \cite[Definition 2.19.(2)]{FSO}, where we fix a closed embedding 
\[\Spec R\hookrightarrow \C^l\]
for some $l\in \Z_{>0}$.

The following lemma follows directly from constructibility and Nakayama Lemma. We leave its proof for interested readers. 
\begin{lemma}\label{lm:relconstrnakayam}
Let $\cF^\bullet$ be a $\C$-constructible complex of $R$-modules on $X$. Then 
the following are equivalent:
\begin{enumerate}
    \item $\cF^\bullet=0$ in $D^b_c(R_X)$
    \item $\cF^\bullet\otimes^{\bL}_R R/m=0$ in $D^b_c(\C_X)$ for every maximal ideal $m\in \Spec R$
    \item $\wt\cF^\bullet=0$ in $\wt{D^b_c}(R_X)$.
\end{enumerate}
\end{lemma}
%\begin{proof}
%It is sufficient to assume $\cF$ to be a constructible sheaf of $R$-modules on $X$. By constructibility and Nakayama Lemma, (1) and (2) are equivalent. One can easily see that $\cF=0\Leftrightarrow \cF_m=0$ for every maximal ideal $m\in \Spec R$ and that  
%\[\textup{$\wt\cF=0\Leftrightarrow \eta_m^{-1}\wt\cF=0$ for every maximal ideal $m\in \Spec R$},\] where $$\eta_m\colon X\times \Specan R_m \hookrightarrow X\times \Specan R$$ is the natural embedding with $\Specan R_m$ the analytic germ. Since $\sO_{\Specan R_m}$ is faithfully flat over $R_m$, we get (1)$\Leftrightarrow$(3).
%\end{proof}

The lemma above immediately implies:
\begin{coro}\label{cor:supprelccsh}
If $\cF^\bullet\in D^b_c(R_X)$, then $\supp_{\Spec R}\cF^\bullet$ is algebraically closed locally over $X$, i.e. locally on a relatively compact open subset $W$ of $X$, 
\[\supp_{\Spec R}\cF^\bullet|_W \subseteq \Spec R\]
is a closed algebraic subvariety. 
\end{coro}

\bibliographystyle{amsalpha}
\bibliography{mybib}

\providecommand{\bysame}{\leavevmode\hbox to3em{\hrulefill}\thinspace}
\providecommand{\MR}{\relax\ifhmode\unskip\space\fi MR }
\providecommand{\MRhref}[2]{%
  \href{http://www.ams.org/mathscinet-getitem?mr=#1}{#2}
}
\providecommand{\href}[2]{#2}
\begin{thebibliography}{BVWZ21b}

\bibitem[BBDG18]{BBDG}
Alexander Beilinson, Joseph Bernstein, Pierre Deligne, and Ofer Gabber, \emph{Faisceaux pervers}, Soci{\'e}t{\'e} math{\'e}matique de France, 2018.

\bibitem[Bei87]{Bgluep}
Alexander Beilinson, \emph{How to glue perverse sheaves}, {$K$}-theory, arithmetic and geometry ({M}oscow, 1984--1986), Lecture Notes in Math., vol. 1289, Springer, Berlin, 1987, pp.~42--51. \MR{923134}

\bibitem[BG12]{BG}
Alexander Beilinson and Dennis Gaitsgory, \emph{A corollary of the b-function lemma}, Selecta Mathematica \textbf{18} (2012), no.~2, 319--327.

\bibitem[Bj{\"o}93]{Bj}
Jan-Erik Bj{\"o}rk, \emph{Analytic {${\mathcal D}$}-modules and applications}, Mathematics and its Applications, vol. 247, Kluwer Academic Publishers Group, Dordrecht, 1993. \MR{1232191}

\bibitem[BL94]{BL}
Joseph Bernstein and Valery Lunts, \emph{Equivariant sheaves and functors}, Lecture Notes in Mathematics, vol. 1578, Springer-Verlag, Berlin, 1994. \MR{1299527}

\bibitem[BLSW17]{BLSW}
Nero Budur, Yongqiang Liu, Luis Saumell, and Botong Wang, \emph{Cohomology support loci of local systems}, Michigan Mathematical Journal \textbf{66} (2017), 295--307.

\bibitem[Bry86]{Bry}
Jean-Luc Brylinski, \emph{Transformations canoniques, dualit{\'e} projective, th{\'e}orie de {L}efschetz, transformations de {F}ourier et sommes trigonom{\'e}triques}, Ast{\'e}risque \textbf{140} (1986), no.~141, 3--134.

\bibitem[Bud15]{Budur}
Nero Budur, \emph{Bernstein-sato ideals and local systems}, Annales de l'Institut Fourier \textbf{65} (2015), no.~2, 549--603 (en).

\bibitem[BVWZ21a]{BVWZ}
Nero Budur, Robin Veer, Lei Wu, and Peng Zhou, \emph{Zero loci of {B}ernstein-{S}ato ideals}, Inventiones mathematicae (2021), 1--28.

\bibitem[BVWZ21b]{BVWZ2}
Nero Budur, Robin van~der Veer, Lei Wu, and Peng Zhou, \emph{Zero loci of bernstein-sato ideals-ii}, Selecta Mathematica \textbf{27} (2021), no.~3, 32.

\bibitem[BW15]{BWcjl}
Nero Budur and Botong Wang, \emph{Cohomology jump loci of quasi-projective varieties}, Ann. Sci. \'{E}c. Norm. Sup\'{e}r. (4) \textbf{48} (2015), no.~1, 227--236. \MR{3335842}

\bibitem[BW17]{BW}
Asilata Bapat and Robin Walters, \emph{The strong topological monodromy conjecture for {W}eyl hyperplane arrangements}, Math. Res. Lett. \textbf{24} (2017), no.~4, 947--954. \MR{3723798}

\bibitem[FF18]{FF18}
Luisa Fiorot and Teresa~Monteiro Fernandes, \emph{t-structures for relative d-modules and t-exactness of the de rham functor}, Journal of Algebra \textbf{509} (2018), 419--444.

\bibitem[FMFS21]{FFS19}
Luisa Fiorot, Teresa Monteiro~Fernandes, and Claude Sabbah, \emph{Relative regular {R}iemann-{H}ilbert correspondence}, Proceedings of the London Mathematical Society \textbf{122} (2021), no.~3, 434--457.

\bibitem[FMFS23]{FMSRRR}
\bysame, \emph{Relative regular riemann--hilbert correspondence {II}}, Compositio Mathematica \textbf{159} (2023), no.~7, 1413--1465.

\bibitem[FS19a]{FS}
Teresa~Monteiro Fernandes and Claude Sabbah, \emph{{R}iemann--{H}ilbert correspondence for mixed twistor d-modules}, Journal of the Institute of Mathematics of Jussieu \textbf{18} (2019), no.~3, 629--672.

\bibitem[FS19b]{FS17}
\bysame, \emph{Riemann-{H}ilbert correspondence for mixed twistor-modules}, Journal of the Institute of Mathematics of Jussieu \textbf{18} (2019), no.~3, 629--672.

\bibitem[Gin86]{Gil}
Victor Ginsburg, \emph{Characteristic varieties and vanishing cycles}, Inventiones Mathematicae \textbf{84} (1986), no.~2, 327--402. \MR{833194}

\bibitem[Gyo93]{Gyo}
Akihiko Gyoja, \emph{Bernstein-sato’s polynomial for several analytic functions}, Journal of Mathematics of Kyoto University \textbf{33} (1993), no.~2, 399--411.

\bibitem[Har13]{HartsAG}
Robin Hartshorne, \emph{Algebraic geometry}, vol.~52, Springer Science \& Business Media, 2013.

\bibitem[HTT08]{HTT}
Ryoshi Hotta, Kiyoshi Takeuchi, and Toshiyuki Tanisaki, \emph{{$D$}-modules, perverse sheaves, and representation theory}, Progress in Mathematics, vol. 236, Birkh\"{a}user Boston, Inc., Boston, MA, 2008, Translated from the 1995 Japanese edition by Takeuchi. \MR{2357361}

\bibitem[Kas83]{KasV}
Masaki Kashiwara, \emph{Vanishing cycle sheaves and holonomic systems of differential equations}, Algebraic geometry ({T}okyo/{K}yoto, 1982), Lecture Notes in Math., vol. 1016, Springer, Berlin, 1983, pp.~134--142. \MR{726425}

\bibitem[Kas03]{Kasbook}
\bysame, \emph{{$D$}-modules and microlocal calculus}, Translations of Mathematical Monographs, vol. 217, American Mathematical Society, Providence, RI, 2003, Translated from the 2000 Japanese original by Mutsumi Saito, Iwanami Series in Modern Mathematics. \MR{1943036}

\bibitem[Kas77]{KasBf}
\bysame, \emph{{$B$}-functions and holonomic systems. {R}ationality of roots of {$b$}-functions}, Invent. Math. \textbf{38} (1976/77), no.~1, 33--53. \MR{0430304}

\bibitem[KS13]{KSbook}
Masaki Kashiwara and Pierre Schapira, \emph{Sheaves on manifolds: With a short history.{\guillemotleft}{L}es d{\'e}buts de la th{\'e}orie des faisceaux{\guillemotright}, by {C}hristian {H}ouzel}, vol. 292, Springer Science \& Business Media, 2013.

\bibitem[Lau83]{Laumon}
G{\'e}rard Laumon, \emph{Sur la cat\'{e}gorie d\'{e}riv\'{e}e des {${\mathcal D}$}-modules filtr\'{e}s}, Algebraic geometry ({T}okyo/{K}yoto, 1982), Lecture Notes in Math., vol. 1016, Springer, Berlin, 1983, pp.~151--237. \MR{726427}

\bibitem[Mai16]{Maihyp}
Philippe Maisonobe, \emph{L'{I}d{\'e}al de {B}ernstein d'un arrangement libre d'hyperplans lin{\'e}aires}, preprint arXiv:1610.03356 (2016).

\bibitem[Mai23]{Mai}
\bysame, \emph{Filtration {R}elative, l'{I}d{\'e}al de {B}ernstein et ses pentes}, Rendiconti del Seminario Matematico della Universita di Padova \textbf{150} (2023).

\bibitem[Mal83]{MalV}
B.~Malgrange, \emph{Polyn\^{o}mes de {B}ernstein-{S}ato et cohomologie \'{e}vanescente}, Analysis and topology on singular spaces, {II}, {III} ({L}uminy, 1981), Ast\'{e}risque, vol. 101, Soc. Math. France, Paris, 1983, pp.~243--267. \MR{737934}

\bibitem[Mal04]{MalIr}
Bernard Malgrange, \emph{On irregular holonomic d-modules}, {\'E}l{\'e}ments de la th{\'e}orie des syst{\`e}mes diff{\'e}rentiels g{\'e}om{\'e}triques. S{\'e}minaires \& Congr{\`e}s \textbf{8} (2004), 391--410.

\bibitem[MFS13]{FSO}
Teresa Monteiro~Fernandes and Claude Sabbah, \emph{On the de {R}ham complex of mixed twistor {$\mathscr D$}-modules}, Int. Math. Res. Not. IMRN (2013), no.~21, 4961--4984. \MR{3123671}

\bibitem[NS96]{NSI}
Nitin Nitsure and Claude Sabbah, \emph{Moduli of pre-{$\mathcal D$}-modules, perverse sheaves and the {R}iemann-{H}ilbert morphism-{I}}, Math. Annalen \textbf{306} (1996), 47--73.

\bibitem[Sab87a]{Sab}
Claude Sabbah, \emph{Proximit\'{e} \'{e}vanescente. {I}. {L}a structure polaire d'un {${\mathscr D}$}-module}, Compositio Math. \textbf{62} (1987), no.~3, 283--328. \MR{901394}

\bibitem[Sab87b]{Sab2}
\bysame, \emph{Proximit\'{e} \'{e}vanescente. {II}. \'{E}quations fonctionnelles pour plusieurs fonctions analytiques}, Compositio Math. \textbf{64} (1987), no.~2, 213--241. \MR{916482}

\bibitem[Sab90]{Sab90}
\bysame, \emph{Modules d’{A}lexander et {$\mathcal D$}-modules}, Duke Mathematical Journal \textbf{60} (1990), no.~3, 729--814.

\bibitem[Sch12]{Schpbook}
Pierre Schapira, \emph{Microdifferential systems in the complex domain}, vol. 269, Springer Science \& Business Media, 2012.

\bibitem[{Sta}18]{stacks-project}
The {Stacks Project Authors}, \emph{\textit{Stacks Project}}, \url{https://stacks.math.columbia.edu}, 2018.

\bibitem[Wu21]{Wubf}
Lei Wu, \emph{On the comparison of nearby cycles via $ b $-functions}, Journal of Singularities \textbf{23} (2021), 92--106.

\bibitem[Wu22]{Wuch}
\bysame, \emph{Characteristic cycles associated to holonomic {$\mathcal D$}-modules}, Mathematische Zeitschrift \textbf{301} (2022), no.~2, 2059--2098.

\bibitem[WZ21]{WZ}
Lei Wu and Peng Zhou, \emph{Log {D}-modules and index theorems}, Forum of Mathematics, Sigma \textbf{9} (2021).

\end{thebibliography}

\end{document}